\documentclass[11pt,reqno]{amsart} 

\usepackage[utf8]{inputenc} 
\usepackage{bbm}

\usepackage{xcolor}
\usepackage[normalem]{ulem} 
\usepackage{soul} 

\usepackage[margin=1in]{geometry} 

\usepackage{graphicx} 
\usepackage{float} 

 \usepackage[parfill]{parskip} 
 
\usepackage{booktabs} 
\usepackage{array} 
\usepackage{paralist} 
\usepackage{verbatim} 
\usepackage{subfig} 
\usepackage{mathrsfs}
\usepackage{amssymb}
\usepackage{amsthm}
\usepackage{amsmath,amsfonts,amssymb,esint}
\usepackage{graphics,color}
\usepackage{enumerate, enumitem}
\usepackage{mathtools,centernot}
\usepackage{cases}
\usepackage{amsrefs}
\usepackage{bbm}
\usepackage{xfrac}

\usepackage{bookmark}

\newtheorem{theorem}{Theorem}[section]
\newtheorem{lemma}[theorem]{Lemma}
\newtheorem{corollary}[theorem]{Corollary}
\newtheorem{definition}[theorem]{Definition}
\newtheorem{proposition}[theorem]{Proposition}
\newtheorem{remark}[theorem]{Remark}

\numberwithin{equation}{section}

\newcommand{\norm}[1]{\left\|#1\right\|}

\newcommand{\T}{\ensuremath{\mathbb{T}}}
\newcommand*{\R}{\ensuremath{\mathbb{R}}}

\newcommand*{\N}{\ensuremath{\mathbb{N}}}

\newcommand{\eps}{\varepsilon}
\newcommand*{\tr}{\ensuremath{\mathrm{tr\,}}}

\usepackage{color, graphicx}
\usepackage{mathrsfs, dsfont}

\usepackage[]{hyperref}
\hypersetup{
    colorlinks=true,       
    linkcolor=red,          
    citecolor=blue,        
    filecolor=red,      
    urlcolor=cyan           
}

\def\dist{\mathop{\rm dist}\nolimits}    
\def\div{\mathop{\rm div}\nolimits}    
\def\dim{\mathop{\rm dim}\nolimits}
\def\spt{\mathop{\rm Spt}\nolimits} 
\def\tr{\mathop{\rm Tr}\nolimits} 
\def\Lip{\mathop{\rm Lip}\nolimits}  

\title{Intermittency and lower dimensional dissipation in incompressible fluids}

\author{Luigi De Rosa}
\address{Department Mathematik Und Informatik, Universit\"at Basel, CH-4051 Basel, Switzerland}
\email{luigi.derosa@unibas.ch}

\author{Philip Isett}
\address{Department of Mathematics, California Institute of Technology, Pasadena CA-91125, USA}
\email{isett@caltech.edu}

\date{\today}

\begin{document}

\begin{abstract}
In the context of incompressible fluids, the observation that  turbulent singular structures  fail to be space filling is known as ``intermittency'' and it has strong experimental foundations.  Consequently, as first pointed out by Landau, real turbulent flows do not satisfy the central assumptions of homogeneity and self--similarity in the K41 theory, and the K41 prediction of structure function exponents $\zeta_p=\sfrac{p}{3}$ might be inaccurate.  In this work we prove that, in the inviscid case, energy dissipation that is lower--dimensional in an appropriate sense implies deviations from the K41 prediction in every $p-$th order structure function for $p>3$.  By exploiting a Lagrangian--type Minkowski dimension that is very reminiscent of the Taylor's \emph{frozen turbulence} hypothesis, our strongest upper bound on $\zeta_p$ coincides with the $\beta-$model proposed by Frisch, Sulem and Nelkin in the late 70s, adding some rigorous analytical foundations to the model.  More generally we explore the relationship between dimensionality assumptions on the dissipation support and restrictions on the $p-$th order absolute structure functions. This approach differs from the current mathematical works on intermittency by its focus on geometrical rather than purely analytical assumptions.

The proof is based on a new local variant of the celebrated Constantin-E-Titi argument that features the use of a third order commutator estimate, the special double regularity of the pressure, and mollification along the flow of a vector field. 
\end{abstract}

\maketitle

\par
\medskip\noindent
{\textbf{MSC (2020):} 		35Q31 - 35D30 - 	76F02 - 28A80.
\par
}

\section{Introduction}
In any spatial dimension $d\geq 2$ we will consider the incompressible Euler equations 
\begin{equation}\label{E}
\left\{\begin{array}{l}
\partial_t v+\div (v\otimes v)  +\nabla p =0\\
\div v = 0,
\end{array}\right.
\end{equation}
on $\Omega\times (0,T)$, where the spatial set $\Omega$ is either the $d-$dimensional torus $\T^d$ or the whole space $\R^d$.

A classical theorem of Constantin-E-Titi \cite{CET94} following \cite{E94}, which confirmed an already quite rigorous\footnote{See \cite{ES06} for an extensive description of Onsager's contributions to the modern theory of fully developed turbulence.} prediction of Onsager \cite{O49}, is that an Euler flow that fails to conserve the kinetic energy 
$$
e_v(t):=\frac{1}{2}\int_{\Omega} |v(x,t)|^2\,dx
$$
cannot have ``more than $\sfrac{1}{3}$ of a derivative in $L^3$'', or more precisely it cannot belong to a Besov class $L_t^3( B_{3,\infty}^{\theta})$ for any $\theta > \sfrac{1}{3}$ (see Section \ref{Sec:tools} for a definition of Besov norms and \cite{DebGS17} for a survey of related results).  The resolution of the Onsager conjecture over the past few decades has confirmed that the exponent $\sfrac{1}{3}$ is sharp \cites{isettOnsag,BdLSV17,Isett17}.  The aim of this paper is to prove theorems of Constantin-E-Titi type that connect to the phenomenon of intermittency and lower dimensional dissipation in turbulent flows.  
We start by explaining the relevant background and giving a rough statement of our strongest result (Theorem \ref{t:rough} below), while we postpone to Section \ref{s:precise_results} all the rigorous statements.

\subsection{Theoretical background and main physical result.}
Physically, the Constantin-E-Titi theorem has the interpretation of a ``singularity theorem''.  It implies that ($L^2$ compact) sequences of Navier-Stokes solutions with viscosity tending to zero must exhibit unbounded growth of the $L_t^3 (B_{3,\infty}^{\theta})$ norm for every $\theta > \sfrac{1}{3}$ if the rate of kinetic energy dissipation remains bounded from below \cite{DrEy19}.  Experimentally, there is considerable evidence for the persistence of energy dissipation at low viscosity (which is known as the ``$0-$th law of turbulence'', see \cite{V15} for a recent review), and $L^2$ compactness is a consequence of observed scaling of the energy spectrum \cite{CG12}.  Persistence of energy dissipation in the vanishing viscosity limit has also been recently shown to exist mathematically for the forced Navier-Stokes equations \cite{BD22} and the passive--scalar advection (scalar theory of turbulence) \cite{CCS22}.
Consequently, singular Euler flows that arise in the inviscid limit and dissipate energy form natural objects to model the behavior that can occur in the vanishing viscosity limit of fully developed turbulence.  The critical exponent $\sfrac{1}{3}$ that plays a pivotal role in the Constantin-E-Titi singularity theorem is also the regularity predicted by the K41 theory of turbulence \cite{K41} for a (dissipative) turbulent flow.

More precisely, K41 predicts the size of ``absolute structure functions'', which measure the average velocity fluctuations at scale $|\ell|$ in the inertial range of length scales (where viscosity presumably plays no role).  The K41 prediction is a scaling law 
$$\langle |v(x+\ell) - v(x)|^p \rangle^{\sfrac{1}{  p}} \sim |\ell|^{\sfrac{\zeta_p}{p}},$$
with $\sfrac{\zeta_p}{  p} = \sfrac{1}{3}$ for all $p \geq 1$,  where  $\langle \cdot \rangle$ is some relevant averaging procedure, space, time or ensemble.  When we interpret the averaging as a spatial average, this scaling law translates to exact\footnote{Consequently, we interpret $\sfrac{\zeta_p}{p}$ mathematically to be the supremum of $\theta$ such that $v \in B_{p,\infty}^\theta$.} $B_{p,\infty}^{\sfrac{\zeta_p}{  p}} = B_{p,\infty}^{\sfrac{1}{3}}$ regularity of the limiting Euler flows, for which all small length scales lie in the inertial range.  Remarkably, however, this scaling law appears to be inconsistent with experiments (at least for $d=3$), indicating that the homogeneity and self--similarity assumption of the K41 theory appears to be false for real turbulent flows.  Shortly after the appearance of the K41 theory, the possibility that the scaling of structure functions could deviate from K41 was first pointed out by Landau \cite{F95}*{Section 6.4}, leading Kolmogorov to revisit his theory \cite{K62} in 1962
(see \cite{frisch1991global} for a historical account).  Indeed, what appears to hold is that for $p$ noticeably larger than $3$, the averages are larger than K41 predicts, so that the exponent $\sfrac{\zeta_p}{  p} < \sfrac{1}{3}$, while for $p$ noticeably smaller than $3$ the averages are smaller than K41 predicts ($\sfrac{\zeta_p}{  p}> \sfrac{1}{3}$). We refer to \cite{ISY20} for a very recent numerical evidence of this fact. The exponent $p = 3$ has a special role in turbulence theory relating to the Kolmogorov $\sfrac{4}{5}$ law -- an exact result relating the energy dissipation of a solution to its third order (signed) structure function. 
Further readings on some aspects of intermittency (or lack thereof) for $p=3$ can be found in \cites{Dr22,E07}.

What the strongest main theorem of this article shows (Theorem \ref{t:main_lagr_dim}) is the following (informal) statement for Euler flows whose energy dissipation fails to be spatially homogeneous.  Here we emphasize that the relevant regularity threshold is strictly below $\sfrac{1}{3}$ and that the relevant range of integrability exponents is the range $p > 3$.

\begin{theorem}[Rough inviscid statement]\label{t:rough}
If the dissipation of energy of an Euler flow on $\R^d$ or $\T^d$ is nontrivial and supported on a lower dimensional set  (in the appropriate sense), then there is a linear upper bound on absolute structure function exponents for all $p > 3$ that lies below the K41 prediction $\sfrac{\zeta_p}{ p} = \sfrac{1}{3}$.  More precisely, let $\theta_p = \sfrac{1}{3} - (d-\gamma)\sfrac{(p-3)}{3p}$, where $\gamma\in [0,d)$ is an upper bound on the dimension of the dissipation set in the sense of Theorem~\ref{t:main_lagr_dim}.  Then, if $\theta\in (0,1)$,  the velocity field cannot be of class $L_t^p (B_{p,\infty}^{\theta})$ for any $\theta>\theta_p$. 
\end{theorem}

In terms of the structure functions of an Euler flow, this theorem is a quantitative upper bound of the form $\sfrac{\zeta_p}{p} < \sfrac{1}{3}$ for all exponents $p > 3$ under the assumption of nontrivial, lower--dimensional energy dissipation.  In particular, these absolute structure functions become larger than K41 predicts.  Note that $\theta_p> 0$ for all $p\geq 3$ if and only if $\gamma \in [d-1,d]$.

Beyond the statement of this theorem, our proof provides a machinery for turning certain assumptions about lower--dimensionality of the dissipation measure into conclusions about the Besov regularity of the solutions. The strongest versions of this argument apply to the full range $p > 3$ as the above theorem states, but weaker assumptions lead to weaker conclusions.  Similarly to the Constantin-E-Titi theorem, our result can equivalently be stated as criteria for the conservation of energy (see indeed the rigorous statement given in Theorem \ref{t:main_lagr_dim} below): If the dissipation is lower--dimensional in the appropriate sense and the regularity of the solution is $L_t^p (B_{p,\infty}^{\theta_p+\epsilon})$, then the energy is conserved. However, our proof contains some new ingredients compared to the Constantin-E-Titi argument and its known generalizations that are related to the analysis of local energy dissipation (see \eqref{Local_energy} below): One has to deal with trilinear commutators, the improved double regularity of the pressure from \cites{CD18,CDF20,Is2013} plays a crucial role, and the construction of a specific cut--off is needed to test the local energy balance \eqref{Local_energy} since it can only be interpreted in a distributional sense in our Besov class of solutions. The definition of the cut--off is indeed the key point in order to handle different notions of dimension.

The statement of the above theorem (as well as our other main theorem, Theorem~\ref{t:main}) is strongly motivated by the phenomenon of ``intermittency'' in turbulence, which is the property that turbulence empirically fails to be space--filling and can even be concentrated on lower dimensional sets (see \cite{BJPV98,MS87,FP85} and also \cite{S81} for some numerical evidence). This ``patchiness'' of turbulence and how it relates to anomalies in the exponents $\zeta_p$ (see for instance the discussion in \cite{frisch1991global})  still represents an important and challenging problem in mathematical physics. What our main theorem establishes is a quantitative downward deviation from the K41 theory for absolute structure functions of order greater than $3$ that relates the dimension of the dissipation to the structure function exponent $\zeta_p$.  For other interesting mathematical works addressing intermittency we refer to \cites{ET99,CS14, LS15,LS17, Shv18, BMNV21, NV22, CS22}. The latter works focus on a \emph{purely analytical} approach to intermittency, while in this note we are exploring a more \emph{geometrical} one, which we hope gives some different/new insights on the topic.

Our quantitative relationship between the upper bound on the dimension and the upper bound on the structure functions exponents $\zeta_p$ turns out to be exactly the one proposed by Frisch, Nelkin and Sulem in \cite{FSN78} in their ``$\beta-$model''.  The $\beta-$model has been among the first simple attempts to correct the K41 theory by introducing nonhomogeneity in modeling turbulence.  Indeed, in the physically relevant three dimensional case, they assume that at each stage of the the Richardson cascade process, the total volume occupied by the eddies decreases by a fraction $\beta<1$, and they derive a prediction consistent with Theorem~\ref{t:rough}:
\begin{equation}\label{bound_betamodel}
\zeta_p=\frac{p}{3}-(3-D_\beta)\frac{p-3}{3},
\end{equation}
where $D_\beta$ is what they call the ``intermittency dimension'' or ``self--similarity dimension,'' motivated by the nomenclature in the work of Mandelbrot \cite{Man75}.  The parameter $D_\beta$ describes the fraction of the space in which an appreciable part of the energetic excitation accumulates along the cascade process. As stated in \cite{AGHA84} the $\beta-$model (with $D_\beta\simeq 2.8$) \eqref{bound_betamodel} fits experimental data rather well for not too large values of $p$.  We refer to \cite[Chapter 8]{F95} for a more detailed discussion about the different fractal and multifractal intermittency models that have been proposed to correct K41.

While our proof provides a general machinery to analyze the link between lower dimensional dissipation and Besov regularity of the solution, we limit ourselves to two main theorems (Theorem \ref{t:main} and Theorem \ref{t:main_lagr_dim}), together with a corollary (see Corollary \ref{t:main_lagr}), that concern spatial intermittency.  These theorems use different notions of dimension and therefore lead to different conclusions.  For future work it is also  of interest to consider slightly different hypotheses and different notions of dimension (for instance a Frostman or Hausdorff one). 
The ``Lagrangian--type Minkowski dimension'' that we use in Theorem \ref{t:main_lagr_dim} (see Definition \ref{Lagr_stab_Min_dim}) is strongly motivated by Taylor's \emph{frozen turbulence} hypothesis \cite{T38} which asserts that ``all turbulent eddies are advected by the mean streamwise velocity, without changes in their properties''. See \cite{Moi09,ZH81} for more recent empirical tests of the Taylor hypothesis.

The two main theorems (Theorem \ref{t:main} and Theorem \ref{t:main_lagr_dim}) emphasize the role of intermittency in {\it space}, as the dissipative set is allowed to (but does not have to) exist for all times and concentrate on lower dimensional sets in space only. In addition to these, we prove a result (Theorem \ref{t:en_cons_in_time}, similar to \cite[Lemma 2.1]{DH21}) that addresses the implications of intermittency in time.  Physically, this result has the following consequence: If one believes the K41 prediction for the {\it third} order structure function $\zeta_3 = 1$ to be exact in the inviscid limit, and also in the $0-$th law that energy dissipates independent of viscosity, then the time support of the limiting dissipation must occupy the whole time interval of existence, leaving no gaps in time where dissipation does not occur.  This result gives further motivation to study the consequences of dissipation that is lower dimensional in space but persists for all time, which is allowed by the two main theorems, Theorems \ref{t:main} and \ref{t:main_lagr_dim}.  More generally, the time intermittency result quantitatively links the gap between $\zeta_3$ and $1$ with the second order structure function $\zeta_2$ and the dimension of the set of dissipation times.

We remark that adding an external force on the right hand side of the first equation in \eqref{E} of the kind $f\in L^q_{x,t}$ does not affect our spatial intermittency analysis,  and it can be easily incorporated in our estimates (see Remark \ref{r:ext_force} for details). However, since this would have unnecessarily burdened the statements, we prefer not mention it in our main theorems.

Before going into the precise definitions and statements of the results proved in the current paper, let us see the implication of Theorem \ref{t:rough}  for incompressible flows in the infinite Reynolds number limit and how it fits in the existing mathematical literature.

\subsection{Interpretation in the inifinite Reynolds number limit}

Let us point out that, in the same way that the Constantin-E-Titi theorem  extends to a singularity theorem for a vanishing viscosity limit of solutions to Navier-Stokes (see for instance \cite{DrEy19}), our result also implies a statement about the inviscid limit.  

\begin{theorem}[Rough vanishing viscosity statement]\label{t:rough_visc}
    Let $\{v^\nu\}_{\nu>0}$ be a smooth sequence of solutions to the incompressible Navier-Stokes equations on $\T^d$ or $\R^d$ with $\nu\rightarrow 0$. Assume that $v^\nu \rightarrow v$ in $L^2_{x,t}$ and that the dissipation $\nu|\nabla v^\nu|^2 $ accumulates on a nontrivial lower dimensional set with dimension upper bound $\gamma\in [0,d)$ in the sense of Definition \ref{Lagr_stab_Min_dim}. Then for $p>3$, by setting $\theta_p = \sfrac{1}{3} - (d-\gamma)\sfrac{(p-3)}{3p}$,  and $\theta\in (0,1)$, it must hold that
    $$
\lim_{\nu\rightarrow 0} \left\| v^\nu\right\|_{L^p_t(B^{\theta}_{p,\infty})}=+\infty, \quad \forall \theta>\theta_p.
    $$
\end{theorem}
The previous theorem shows that, under  the assumptions of $L_{x,t}^2$ compactness (which is mild from the point of view of the observations in \cite{CG12} but in general is not rigorously justified), all the intermittent Besov norms $L^p_t(B^{\theta_p+\eps}_{p,\infty})$ must blowup in the inviscid limit, if the limiting dissipation is nontrivial and lower dimensional. Let us remark that in this context, nontrivial dissipation is usually stated as 
\begin{equation}
    \label{0th_law}
    \liminf_{\nu\rightarrow 0}\nu\int_0^T\int_{\Omega}|\nabla v^\nu|^2\,dxdt>0,
\end{equation}
which indeed might be seen as a precise mathematical formulation of the $0-$th law of turbulence.

Leaving out for a moment whether or not this is plausible for real turbulent flows, when $\gamma=d-1$ Theorem \ref{t:rough_visc}  implies that no global fractional regularity measured in the $L_x^\infty$ scale can be retained in the infinite Reynolds number limit if the dissipation is nontrivial and concentrated on a codimension$-1$ set.  That is, more precisely, $\theta_p=\sfrac{1}{p}\rightarrow 0$ as $p\rightarrow \infty$, which is equivalent to $\zeta_p=1$ for all $p$.  We refer to \cite{ISY20} for the most updated numerical simulations predicting actual values of $\zeta_p$, for $p\lesssim 12$, with plausible saturation $\zeta_p\rightarrow const.$ for $p\rightarrow\infty$.    Codimension$-1$ singularities also appear in shocks--types singularities for various compressible models. Much of the authors work contained in this paper carry over to this setting, with suitable  additional assumptions,  and it would show the authors' conclusions are sharp for a certain class of solutions. We emphasise that truly concave and bounded upper bounds on $\zeta_p$, i.e. without imposing $\zeta_p=1$ for all $p$,  can be deduced by our analysis given in Theorem \ref{t:main}. This is again related to codimension$-1$ singularities, but with a possibly weaker notion of dimensionality (see Definition \ref{d:Eul_Min_dim}). Details are given in Remark \ref{r:concave_ub}.

In Theorem \ref{t:lagr_vanish_visc} we will give the rigorous counterpart of the previous result, which relies on the Lagrangian notion of dimension.

\subsection{Comparison with the existing literature.} Let us start by noticing that under the homogeneity assumption, i.e. when $\gamma=d$, the critical threshold given in Theorem \ref{t:rough} (as well as all the results in both space and time we prove in this paper) automatically coincides with the K41 (and thus Onsager) prediction $\theta=\sfrac{1}{3}$, thus providing an honest generalisation of the energy conservation criterion of Constantin-E-Titi to the intermittent setting. Moreover, our Theorem \ref{t:rough} also matches (modulo the definition of dimension) with the energy conservation of $d-$dimensional Vortex Sheets proved in \cite{Sv09} by R. Shvydkoy.  See also the recent work \cite{DI23} where the first named author together with M. Inversi generalized such energy conservation, with no assumptions on the structure of the singular set. Vortex Sheets are intermittent incompressible flows in the class $L^\infty\cap BV$ (which by interpolation belong to $B^{\sfrac{1}{p}}_{p,\infty}$ for all $p\geq 1$) which are smooth outside a regular enough codimension$-1$ surface (the sheet) across which they have a jump discontinuity in the transversal direction. Indeed our result proves that, when singularities form at most a codimension$-1$ set, the threshold that determines energy conservation is $\theta_p=\sfrac{1}{p}$, which is exactly the critical regularity of a Vortex Sheet. We refer to Remark \ref{r:vortexsheets} for a more rigorous discussion. Note that such a threshold $\theta_p=\sfrac{1}{p}$ coincides also with codimension$-1$ Burgers shocks (that formally corresponds to $\gamma=0$ in $1d$ Burgers) which, contrary to Vortex--Sheets, are also consistent with a nontrivial energy dissipation \cite{F95}*{Pages 142-143}.

A non--quantitative precursor to Theorem \ref{t:rough} was given in \cite{Isett17}, where the second author proved (by a rather different argument) that an energy--nonconserving solution cannot be of class $L_t^p (B_{p,\infty}^{\sfrac{1}{3}})$ for any $p > 3$ if its dissipation is supported on a set of space--time Lebesgue measure zero.  Moreover, our result drastically improves \cite{DH22} (see also \cite{CKS97}, which at the best of our knowledge has been the first attempt to analytically/quantitatively analyze lower dimensional singularities) by the first author and S. Haffter by removing any sort of hypothesis on the quantitative smoothness of the velocity field $v$ outside the singular set and moreover by developing a machinery that requires only an assumption on the dimensionality of the support of the dissipation.  
Notice that the set in which the dissipation is supported is clearly always contained the complement of the set in which the velocity is $C^1$.  In contrast to our Theorem \ref{t:rough}, note that the two main theorems of \cite{CKS97,DH22} cannot really be thought as purely intermittency results: 
They both assume some quantitative behaviour of $v$ outside the singular set, which in turn implies that, if the dimension of the singular set is small enough, $v\in L^3_t (B^{\sfrac{1}{3}+}_{3,\infty})$, thus always relying on control of the third order exponent $\zeta_3$.  Instead the results we provide in this paper not only rule out the condition of quantitative smoothness outside a small set, but directly restrict the $p-$th order structure functions for $p>3$ under general conditions while placing no restriction on $\zeta_3$ that implies energy conservation.

Regarding the complementary range $p < 3$, we note that convex integration has been recently applied to produce energy--nonconserving Euler flows even beyond the $\sfrac{\zeta_p}{  p} = \sfrac{1}{3}$ threshold in \cite{BMNV21,NV22}, supporting the expectation that dissipation in the vanishing viscosity limit is consistent with structure functions exponents larger than the ones predicted by Kolmogorov.  Note that there is some heuristic consistency between our result and the intermittent Onsager theorem proved in \cite{NV22}. They produce dissipative solutions of Euler belonging to $L^{\infty-}\cap H^{\sfrac{1}{2}-}$ which by interpolation gives $B^{\sfrac{1}{p}-}_{p,\infty}$ for every $p\geq 2$, or equivalently in terms of structure functions exponents $\zeta_p=1$ (up to an $\eps$) for all $p\geq 2$, and moreover their construction indicates, at least heuristically, that spatial singularities are concentrated on a $2-$dimensional set \cite[Remark 1.2]{NV22}. This is indeed consistent with the upper bound given by Theorem \ref{t:rough}, which,  modulo the specific notion of dimension,  in the case $\gamma=2$ (in the three dimensional setting $d=3$) forces $\zeta_p\leq 1$ (at least for all $p\geq 3$) in order to allow the Euler equations to support dissipation. Thus, our result also shows that a $2-$dimensional singular set is the smallest one that can be achieved for the dissipative solutions constructed in \cite{NV22}.  Further discussions on intermittency phenomena for $p<3$ can be found in \cite{F95}.

\subsection*{Acknowledgments} L.D.R. has been partially supported by the 2015 ERC Grant 676675
FLIRT (Fluid Flows and Irregular Transport). P.I. acknowledges the support of a Sloan Fellowship and the NSF grant DMS-2055019.  Moreover, the authors would like to thank Theodore D. Drivas for very inspiring conversations about intermittency and physically relevant phenomenologies in turbulence. 

\subsection*{Data Availability \& Conflict of Interest Statements} Data sharing not applicable to this article as no datasets were generated or analysed during the current study.  All authors declare that they have no conflicts of interest.

\section{Precise Statements}\label{s:precise_results}

Before giving the precise statements of all the intermittency results in both space and time, we recall the energy balance for Euler which will allow us to rigorously define what we mean by \emph{dissipation}, together with its \emph{lower dimensionality}, in our context. 
\subsection{Local energy balance and Duchon-Robert distribution.} From \cite{DR00} we have the following local energy balance 
\begin{equation}
\label{Local_energy}
\partial_t\left( \frac{|v|^2}{2}\right)+\div \left(\left(\frac{|v|^2}{2}+p \right) v \right)=-D^v,
\end{equation}
for some distribution $D^v\in \mathcal{D}'(\Omega\times (0,T))$, known as the Duchon-Robert distribution,  whenever $v\in L^3_{loc}(\Omega\times (0,T))$. The distribution $D^v$ measures the local dissipation of energy of turbulent inviscid flows.  When the solution arises as a limit in $L^3_{x,t}$ of classical or more generally ``dissipative'' (see \cite[Section 3]{DR00} for the definition of a \emph{dissipative solution}) Navier-Stokes solutions $v^\nu$, i.e.
\begin{equation}\label{NS}
\left\{\begin{array}{l}
\partial_t v^\nu+\div (v^\nu\otimes v^\nu)  +\nabla p^\nu -\nu \Delta v^\nu=0\\
\div v^\nu = 0,
\end{array}\right.
\end{equation}
then $D^v$ is necessarily a non-negative distribution, and thus by Riesz's Theorem it is a measure, that is equal to the weak limit  of the local dissipation measures of the Navier-Stokes flows if $v^\nu$ is regular enough, that is
\begin{equation}
   \label{DR_equals_diss}
    \lim_{\nu\rightarrow 0} \nu |\nabla v^\nu|^2=D^v, \quad \text{in } \mathcal{D}'(\Omega\times (0,T)),
\end{equation}
as soon as  $v^\nu \rightarrow v$ in $L^3_{x,t}$.
For a general Euler flow in $L^3_{x,t}$ we have the formula
\begin{equation}\label{DR_measure}
D^v=\lim_{\eps\rightarrow 0} R_\eps : \nabla v_\eps \quad \text{in } \mathcal{D}'(\Omega\times (0,T)),
\end{equation}
where $v_\eps=v*\rho_\eps$ is a standard regularisation in space and $R_\eps:=v_\eps\otimes v_\eps-(v\otimes v)_\eps$ is the Reynolds--type stress tensor arising from averaging at scale $\eps$ the only nonlinear term in the equation (see for instance \eqref{local_energy_eps} below). For further formulas of the dissipation measure and its connection with some Lagrangian aspects of turbulence we refer to \cite{Dr19}.

It is known that $D^v$ is the trivial distribution $D^v\equiv 0$ whenever $ v\in L^3((0,T);B^{\theta}_{3,\infty}(\Omega))$ for $\theta>\sfrac{1}{3}$ \cites{CET94,DR00,CCFS08}.  Technically the original proof of \eqref{Local_energy} from \cite{DR00} uses a slightly different approximation for $D^v$ and the proof of the specific formula \eqref{DR_measure} has been given by the second author in \cite{Isett17}, which also shows that $\spt_{x,t}{ D^v}$ is contained in the singular support of $v$ relative to the critical conservative space, see \cite[Theorem 4]{Isett17}.

\subsection{Time intermittency}\label{Intro:time} We start with a couple of results whose goal is to investigate the case in which singularities are not time--filling but they instead concentrate on a lower dimensional set. In addition to their novelty in the current literature, they also serve as a motivation to consider the spatial intermittency in which singularities can happen for every time as discussed at the end of this section.  We analyse the singularities in time in the spirit of \cite[Theorem 1.1]{DH21} in which the lower bound on the dimension of singular times comes as a consequence of the H\"older regularity of the kinetic energy (see for instance \cite{CD18,Is2013}) 
\begin{equation}\label{energy_holdercont}
|e_v(t)-e_v(s)|\leq C|t-s|^{\frac{2\theta}{1-\theta}},
\end{equation}
for solutions $v\in L^\infty((0,T);B^\theta_{3,\infty}(\Omega))$. It is worth mentioning that from \cite{DT21} the energy regularity \eqref{energy_holdercont} is sharp. More precisely most of the solutions to Euler (in the appropriate sense) do not posses any better regularity than \eqref{energy_holdercont}, and fail to have monotonic kinetic energy on every interval. In connection to intermittency in turbulence, this result gave (at least partially) a positive answer to the conjecture from \cite{IsOh17}, indicating that energy dissipation is highly unstable upon the slightest departure from the $\sfrac{1}{3}$ law (or more specifically the regularity $B^{\sfrac{1}{3}}_{3,\infty}$).

The $L^\infty-$in time regularity assumption on $v$ is essential to get the desired H\"older regularity \eqref{energy_holdercont} of the kinetic energy, but, as already claimed in \cite{D20} while investigating the helicity regularity,  a weaker integrability in time would still imply a suitable Sobolev (or Besov) regularity of conserved quantities like the kinetic energy or the helicity. Indeed we have

\begin{proposition}\label{p:en_besov_regular}
Let $\Omega$ be $\T^d$ or the whole space $\R^d$, $d\geq 2$. Let $\theta, \beta\in (0,1)$, $p\geq 3$,  and $v\in L^{p}((0,T); B^\theta_{3,\infty}(\Omega))\cap L^{\sfrac{2p}{3}}((0,T); B^\beta_{2,\infty}(\Omega)) $ be a weak solution of Euler. Then 
$$
\norm{e_v(\cdot+h)-e_v(\cdot)}_{L^{\sfrac{p}{3}}_t}\leq C |h|^{\frac{2\beta}{1-3\theta+2\beta}},
$$
for all $h\neq 0$ and for some constant $C>0$, namely $e_v\in B^{\frac{2\beta}{1-3\theta+2\beta}}_{\sfrac{p}{3},\infty}([0,T])$ and we implicitly assumed that the time integral in the previous formula is computed for all times $t$ such that $t,t+h\in (0,T)$.  In particular, as an $L^1([0,T])$ function, $e_v$ is constant if $\theta >\sfrac13$.
\end{proposition}
The previous proposition generalizes the Sobolev regularity  $e_v\in W^{1,\sfrac{p}{3}}([0,T])$ obtained in \cite[Theorem 1.6]{Is2013} when $\theta=\sfrac13$.  Moreover, we make a distinction between the two spatial regularities ($\theta$ and $\beta$) of the velocity $v$ when measured in the two different integrability scales $L_x^3$ and $L_x^2$, to allow the second order structure function to have higher regularity exponents, which is expected from observable turbulence.  We remark that in the periodic case $\Omega=\T^d$ the assumption $v(t)\in B^\beta_{2,\infty}(\Omega)$ is redundant whenever $\beta\leq \theta$, and one gets
$$
\norm{e_v(\cdot+h)-e_v(\cdot)}_{L^{\sfrac{p}{3}}_t}\leq C |h|^{\frac{2\theta}{1-\theta}}
$$
by only assuming $v\in L^{p}((0,T); B^\theta_{3,\infty}(\Omega))$.  This is clearly not possible when $\Omega=\R^d$, where a suitable spatial regularity $v(t)\in B^\beta_{2,\infty}(\R^d)$ is essential to estimate the quadratic terms that always arise when splitting the energy increments with respect to the energy $e_{v_\eps}$ of the averaged velocity $v_\eps$ (see the estimate \eqref{increment1-3} below).

For sake of clarity let us define what we mean by conservative/non--conservative solutions in this context in which the kinetic energy $e_v$ is not necessarily continuous.
\begin{definition}[Conservative/non--conservative solutions]\label{d:cons_sol}
Let $v\in L^2((0,T)\times\Omega)$ be a weak solution to Euler. We say that $v$ is a conservative solution if $e'_v\equiv 0$ in $\mathcal{D}'((0,T))$, i.e.
\begin{equation}
    \label{e_v_constant_distribution}
    \langle e_v',\varphi\rangle =0,\quad \forall \varphi\in C^\infty_c((0,T)).
\end{equation}
Consequently, we call non--conservative solutions the ones for which \eqref{e_v_constant_distribution} fails.
\end{definition}
Being a conservative solution according to the previous definition means that $e_v$ is a constant in the sense of $L^1([0,T])$ functions, which coincides with the usual definition $e_v(t)=e_v(0)$ for all $t\in [0,T]$ as soon as $v\in C^0([0,T];L^2(\Omega))$. However, we have decided not to add any continuity assumption on the kinetic energy in order to keep the results more general.

In what follows we will denote $\overline \dim_{\mathcal{M}} S_T$ the upper Minkowski dimension of the set $S_T\subset (0,T)$ (see Section \ref{Sec:tools} for the rigorous definition). 
Similarly to \cite[Lemma 2.1]{DH21}, we then get the following

\begin{theorem}\label{t:en_cons_in_time}
Let $\Omega$ be $\T^d$ or $\R^d$, $d\geq 2$.  Let $p\geq 3$, $\theta	\in (0,\sfrac13)$,  $\beta\in (0,1)$ and $v\in L^{p}((0,T);B^\theta_{3,\infty}(\Omega))\cap L^{\sfrac{2p}{3}}((0,T);B^\beta_{2,\infty}(\Omega))$ be a weak solution of Euler. Suppose that there exists a closed set of times $S_T\subset (0,T)$ with $\overline \dim_{\mathcal M}S_T=\gamma <1$ 
such that $v$ is conservative on $S_T^c \times \Omega$.
Then  $v$ conserves the kinetic energy if 
$$
\frac{2\beta}{1-3\theta+2\beta}>1-\frac{p-3}{p}(1-\gamma).
$$
 \end{theorem}
We note that by \cite{CCFS08} a sufficient condition for being conservative outside $S_T$ is that $v$ is locally of class $L_t^3B^{\sfrac13}_{3,c_0}$ on $S_T^c\times \Omega$.
 
The proof of the previous theorem will be given in Section \ref{s:time_intermitt}. As a corollary we obtain the corresponding lower bound on the Minkowski dimension of singular times of non--conservative weak solutions of Euler, which can be equivalently read an upper bound on the regularity exponents $\theta,\beta$.
 \begin{corollary}\label{c:lower_bound_time}
Let $\Omega$ be $\T^d$ or $\R^d$, $d\geq 2$. Let $p\geq 3$, $\theta	\in (0,\sfrac13)$,  $\beta\in (0,1)$ and $v\in L^{p}((0,T);B^\theta_{3,\infty}(\Omega))\cap L^{\sfrac{2p}{3}}((0,T);B^\beta_{2,\infty}(\Omega))$ be a non--conservative weak solution of Euler. Suppose that there exists a closed set of times $S_T\subset (0,T)$ with $\overline \dim_{\mathcal M}S_T=\gamma <1$ 
such that $v$ is conservative on $S_T^c \times \Omega$.  
Then it must hold 
\begin{equation}\label{deviation_theta}
\frac{2\beta}{1-3\theta+2\beta}\leq 1-\frac{p-3}{p}(1-\gamma).
\end{equation}
 \end{corollary}
Note that, as soon as $p>3$ and $\gamma<1$,  inequality \eqref{deviation_theta} implies that dissipative solutions must have their $L^3$ regularity index $\theta<\sfrac13$, no matter the value of $\beta$. The gap between $\theta$ and $\sfrac{1}{3}$ becomes larger when the $L^2$ regularity index $\beta$ is larger, which means that higher values of $\zeta_2$ imply larger downwards deviations from K41 in the third order structure functions -- that is unless, of course, the time support of the dissipation has dimension $\gamma = 1$.  Since a gap where the $L^3$ regularity index $\theta$ lies below $\sfrac{1}{3}$ may be regarded as unphysical, for example in its contrast to the $4/5-$ths law, 
we are motivated to find conditions under which there are no gaps in time having zero dissipation but nonetheless intermittency in the structure functions can be observed.  The main theorems that follow allow for this behavior.

\subsection{Spatial intermittency}
To state our main results concerning lower dimensional singularities in space we will need different notions of Minkowski--type dimension of the dissipative set $S\subset \Omega\times (0,T)$, that is the support of the Duchon-Robert distribution $D^v$ from \eqref{Local_energy}, i.e.
$$
S:=\spt_{x,t} D^v.
$$
In the following we will refer to the support of the distribution $D^v$ as the \emph{dissipation support} and we will denote by $\mathcal{H}^k$ the $k-$dimensional Hausdorff measure, $k\in \N$.

\begin{definition}[Eulerian time stable Minkowski dimension]\label{d:Eul_Min_dim}
Let $\Omega$ be $\R^d$ or $\T^d$. Given a set $S\subset \Omega\times (0,T)$ we will say that $S$ has Eulerian time stable Minkowski dimension at most $\gamma\in [0,d]$, with instability parameter $\beta>0$,  if, by letting $\tau=\delta^\beta$ and defining 
$$
(S)_{\delta,\tau}:=\left\{\left(x+h,t+s\right)\, :\, (x,t)\in S,\, |h|<\delta,\, |s|<\tau \right\},
$$
to be the space--time  $\delta,\tau-$neighbourhood of the set $S$, it holds
\begin{equation}\label{E_Min_dim_gamma_def}
\mathcal{H}^{d+1}((S)_{\delta,\tau}\cap M)\lesssim \delta^{d-\gamma},
\end{equation}
for all $\delta>0$ sufficiently small and every bounded set $M\subset \Omega\times (0,T)$, where the implicit constant in \eqref{E_Min_dim_gamma_def} might depend on $d$, $\beta$, $\gamma$ and $M$, but it is otherwise independent on $\delta$.
\end{definition}
The necessity of introducing the bounded set $M$ in \eqref{E_Min_dim_gamma_def} is to cover the case $\Omega=\R^d$ and allowing the dissipation support to be unbounded (in space).  However, we remark that since we are willing to prove a \emph{local} energy conservation result according to \eqref{Local_energy}, the local properties of the dissipation support are the only ones that really matter.  Clearly when $\Omega=\T^d$ the set $M$ does not play any role and it can simply be removed.  

The parameter $\beta$ is called the ``time instability parameter'' because it quantifies how long in time a covering by cylinders can remain a covering, and thus quantifies how ``rough'' the motion of the set is in time.  Moreover, by the previous definition it readily follows that if a set $S$ has Eulerian Minkowski dimension at most $\gamma$, with time instability parameter $\beta>0$, then the same Eulerian dimension upper bound holds true for every other instability parameter $\beta'>\beta$. 
Similar to the standard Minkowski one, our Eulerian notion of dimension can be reformulated in terms of a box--counting one: The set $S$ can be covered by at most $\tau^{-1}\delta^{-\gamma}$ space--time cylinders $\big\{\mathcal{C}_i^{\delta,\tau}\big\}_i$ of spatial radius $\delta$ and time length $\tau=\delta^\beta$. Notice that all of such cylinders have $\mathcal{H}^{d+1}\big(\mathcal{C}_i^{\delta,\tau}\big)=\tau\delta^d$.

\noindent{\bf Example 1.} An easy example of a set that satisfies the definition above is the Galilean transformation of a $\gamma-$dimensional set.

\noindent{\bf Example 2.}  More generally, consider the image of a $\gamma-$Minkowski dimensional set $\tilde S\subset \T^d$ through the flow of a Lipschitz 
vector field $U\in L^\infty((0,T); \text{Lip}(\Omega))$. More precisely if $\tilde S\subset \T^d$ satisfies 
$$
\mathcal{H}^d\left(\big(\tilde S\big)_\delta\right)\lesssim \delta^{d-\gamma},
$$
then the space--time set
$$
S=\left\{ \left( \phi^U(x,t),t\right)\, :\, x\in \tilde S, \,t\in (0,T)\right\}
$$
has Eulerian time stable Minkowski dimension at most $\gamma$ with instability parameter $\beta= 1$, where $\phi^U=\phi^U(x,t)$ is the flow associated to the vector field $U=U(x,t)$, i.e.
$$
\left\{\begin{array}{l}
\frac{d}{dt}\phi^U(x,t)=U\left(\phi^U(x,t),t\right)\\ \\
\phi^U(x,0)=x\,.
\end{array}\right.
$$
Indeed, consider a covering of $\tilde S$ by $O(\delta^{-\gamma})$ $\delta-$balls.  Evolving the covering, each ball will evolve in a bounded time $T$ to a (possibly distorted) ball of diameter at most $C \delta$ by the Gr\"onwall inequality, where $C\sim e^{T \| \nabla U \|_{L^\infty_{x,t}}}$.  Thus any time $t$ slice of $S$ can be covered by $O(\delta^{-\gamma})$ $\delta-$balls, and if we double the radii of these balls and use the boundedness of $U$, such a covering at time $t$ remains a covering for a time interval of size $\sim \delta$.  This observation combined with the equivalence observed in the following remark shows that $S$ has upper Minkowski dimension at most $\gamma+1$ in space-time and therefore has Eulerian Minkowski dimension at most $\gamma$ with instability parameter $\beta = 1$.

\begin{remark}[Standard Minkowski vs Definition \ref{d:Eul_Min_dim}]\label{r:Eul=Minkow}
Notice that when the instability parameter $\beta\geq 1$ then it readily follows that the Eulerian Minkowski dimension from Definition \ref{d:Eul_Min_dim} is implied by the more common space--time $\gamma+1$ upper Minkowski dimension (see Section \ref{Sec:tools} for the definition), and they actually coincide if $\beta=1$, while for $\beta<1$ the Eulerian--Minkowski dimension is strictly stronger. Indeed, suppose that $S\subset \T^d \times (0,T)$ has upper Minkowski dimension at most $\gamma+1$. Then by \eqref{Mink_asympt}, denoting by $[S]_\delta$ the space--time tubular neighbourhood of width $\delta$ of the set $S$, one has
$$
\mathcal{H}^{d+1}([S]_{\delta})\lesssim \delta^{d+1-(\gamma+1)}=\delta^{d-\gamma},
$$
which, together with $(S)_{\delta,\delta^\beta}\subset [S]_{2\delta} $ for every $\beta\geq 1$, gives that $S$ has Eulerian time stable Minkowski dimension at most $\gamma$ with instability parameter $\beta\geq 1$.  \end{remark}

Theorem \ref{t:main} below can thus be stated purely in terms of the classical Minkowski dimension in space--time, while Eulerian--Minkowski dimension with $\beta < 1$ will be used to formulate the Corollary~\ref{t:main_lagr} below, which has a stronger conclusion.  Note also that the Euler equations have separate scaling symmetries in space and time, so demanding a particular relationship between the two variables in defining dimension may not be entirely natural.  We therefore have introduced the general notion of Eulerian time stable Minkowski dimension at this point to draw attention to the importance of the concept of {\it time stability} (informally, how long do coverings by balls remain coverings), in obtaining our conclusions.

We are now ready to state our first result about spatial intermittency.

\begin{theorem}[Eulerian intermittency]
\label{t:main}
Let $\Omega$ be $\R^d$ or $\T^d$, $d\geq 2$. Let $\gamma\in [0,d]$, $p\in[3,\infty]$, $\theta\in (0,1)$ and $v\in L^p((0,T);B^\theta_{p,\infty}(\Omega))$ be a weak solution of Euler.  Assume that the dissipation set $S\subset \Omega\times (0,T)$  has Eulerian time stable Minkowski dimension at most $\gamma$ with instability parameter $\beta=1$ according to Definition \ref{d:Eul_Min_dim} (or equivalently, in virtue of Remark \ref{r:Eul=Minkow}, $S$ has space--time Minkowski dimension at most $\gamma+1$), namely 
\begin{equation}\label{E_Min_dim_gamma}
\mathcal{H}^{d+1}\left((S)_{\delta,\delta}\cap M\right)\lesssim \delta^{d-\gamma},
\end{equation}
for all $\delta>0$ sufficiently small and some bounded set $M\subset \Omega\times (0,T)$.
Then $D^v\equiv 0$ if 
\begin{equation}\label{lowerbound_theta}
\frac{2\theta}{1-\theta}>1-(d-\gamma)\frac{p-3}{p}.
\end{equation}

\end{theorem}

Since $\sfrac{2\theta}{(1-\theta)}<1\iff\theta <\sfrac{1}{3}  $, it is clear that as soon as $\gamma<d$ (and assuming the dissipation to be nontrivial), the above theorem implies intermittency in every structure function\footnote{The assumption of $L_t^p$ can be replaced by $L_t^3$ or $L_t^q$ with $3 \leq q \leq p$ without changing the argument. The details are given in Remark \ref{r:eulerian_L3_in_time}.} of order $p> 3$. 

\begin{remark}
    \label{r:concave_ub}
    In view of the previous theorem, together with the equivalence observed in Remark \ref{r:Eul=Minkow},  when the space--time Minkowski dimension of the set in which the nontrivial dissipation concentrates is at most $d$, then it must hold $\sfrac{2\theta}{(1-\theta)} \leq\sfrac{3}{p}$, which in terms of structure function exponents reads as 
    $$
\zeta_p\leq \zeta_p^*= \frac{3p}{3+2p}.
$$
Coherently with what recent simulations show \cite{ISY20}, the above expression of $\zeta_p^*$, as a function of $p\geq 3$, is concave, bounded and  gives $\zeta^*_3=1$. Even if our analysis is limited to the range $p\geq 3$, the fact that 
$$
\frac{d}{dp}\zeta_p^*\bigg|_{p=3}=\frac19<\frac{1}{3},
$$
also matches, assuming a $C^1$ behaviour in $p$, with the empirical evidence of upward deviations from K41 predictions for $p<3$.
\end{remark}

In view of Examples 1 and 2 of sets that satisfy Definition~\ref{d:Eul_Min_dim}, we naturally obtain the following dichotomy, which is in fact an intuitive interpretation of Theorem \ref{t:main}.

\noindent{\textbf{Dichotomy} (Interpretation of Theorem \ref{t:main})\textbf{.}} If the dissipation is nontrivial in an incompressible non--viscous flow and is supported on a (spatially) lower dimensional set, then either the dissipation support is moving very rapidly in time (i.e. it cannot be contained in the Lipschitz flow of a set with the same spatial dimension) or there is intermittency in every $p-$th order structure function with $p> 3$.

The above dichotomy yields one reason to look for a (possibly stronger) notion of dimension that will allow for a stronger result.  That is, we may be interested in dissipation sets that do move rapidly in time if they are indeed adapted to a low regularity solution.  Perhaps more importantly, Theorem~\ref{t:main} gives a result that is weaker in terms of intermittency than the prediction of the $\beta$--model that we desire (since $3\theta>\sfrac{2\theta}{(1-\theta)}$), and thus we look for a dimensional assumption that will yield a stronger result.

Intuitively, based on the Taylor hypothesis, we expect the fine scale oscillations of the velocity field, which in turn support the dissipation, to be carried along the coarse scale flow.  We will therefore put forth a definition of dimension that is tailored to the flow map of a smooth approximation to our rough vector field.  

In what follows we will denote by $\Phi^V_s:\Omega\times (0,T)\rightarrow \T^d\times (0,T)$ the flow map associated to the vector field $V:\Omega\times (0,T)\rightarrow \R^d$, namely 
\begin{equation}\label{flow_def}
\Phi^V_s(x,t):=\left(\phi^V_s(x,t),t+s \right),
\end{equation}
where $\phi^V_s(x,t)$ solves 
$$
\left\{\begin{array}{l}
\frac{d}{ds}\phi^V_s(x,t)=V\left(\phi^V_s(x,t),t+s\right)\\ \\
\phi^V_0(x,t)=x\,.
\end{array}\right.
$$

By the classical Cauchy-Lipschitz theory such a flow is well defined whenever $V\in C^0_t(\Lip_x)$. 
We are now ready to introduce the following definition

\begin{definition}[Lagrangian time stable Minkowski dimension]\label{Lagr_stab_Min_dim}
Let $\Omega$ be $\R^d$ or $\T^d$. Let $S\subset \Omega\times (0,T)$ and $v\in L^q((0,T);L^p(\Omega))$ be a coarse incompressible vector field, for some $p,q\in[2,\infty]$. We will say that $S$ has Lagrangian time stable Minkowski dimension  at most $\gamma\in [0,d]$, with instability parameters $\beta_1,\beta_2>0$ relative to $v$, if (there exist implicit constants such that) for all sufficiently small $\delta>0$ there exists an incompressible vector field $V^\delta$ of uniformly Lipschitz class $V^\delta \in C^0([0,T];\Lip(\Omega))$ such that 
\begin{equation}\label{L_stab_beta1}
\left\|v-V^\delta\right\|_{L^q_t(L^p_x)}\lesssim \delta^{\beta_1},
\end{equation}
and, by defining the $\delta,\tau-$Lagrangian tubular neighbourhood 
\begin{equation}
\label{Lagran_neigh}
\mathcal{L}^{V^\delta}(S)_{\delta,\tau}:=\left\{ \Phi_s^{V^\delta}\left(x+ h,t\right)\, : \, (x,t)\in S,\, |s|< \tau , \, |h| < \delta \right\}\,,
\end{equation} 
it holds 
\begin{equation}
\label{L_dim_gamma}
 \mathcal{H}^{d+1} \left(M\cap \mathcal{L}^{V^\delta}(S)_{\delta,\tau} \right)\lesssim \delta^{d-\gamma}\quad \text{if }\, \tau=\delta^{\beta_2}\,,
\end{equation}
for every $\delta>0$ small enough and every bounded set $M\subset \Omega\times (0,T)$, where the implicit constant in \eqref{L_dim_gamma} might depend on $d$, $\beta_1$, $\beta_2$, $\gamma$ and $M$,  but it is otherwise independent of $\delta$.
\end{definition}

As for Definition \ref{d:Eul_Min_dim}, also in this case there is  a hierarchy of (in)stabilities. Indeed, if a set has Lagrangian time stable Minkowski dimension at most $\gamma$ with instability parameters $\beta_1,\beta_2$, then it is also the case for every other couple of parameters $\beta_1'\leq  \beta_1$ and $\beta_2'\geq \beta_2$.  Also, the Lagrangian dimension we gave can be equivalently computed by counting cylinders: To cover the set $S$ we need at most $\tau^{-1}\delta^{-\gamma}$ Lagrangian cylinders $\big\{\mathcal{L}_i^{\delta,\tau}\big\}_i$, generated by a Lipschitz incompressible vector field $V^\delta$ which is $\delta^{\beta_1}$ close to $v$ in the $L^q_t(L^p_x)$ topology, with spatial radius $\delta$ and time length $\tau=\delta^{\beta_2}$. Notice that each such cylinder has $\mathcal{H}^{d+1}\big(\mathcal{L}_i^{\delta,\tau}\big)=\tau\delta^d$.

Let us emphasize that there will always be at least one family $V^\delta$ satisfying the stability requirement \eqref{L_stab_beta1} measured by the parameter $\beta_1$ for the class of velocity fields $v\in L^p_t(B^\theta_{p,\infty})$, with $p\geq 3$, we will consider.  Indeed, one can simply choose $V^\delta=v*\rho_\delta$ to be the spatial mollification of $v$ at scale $\delta$, or even at a smaller length scale $\delta' < \delta$, and by standard mollification estimates (see \eqref{moll1} below) deduce $\beta_1\geq \theta$ in \eqref{L_stab_beta1}, which is actually satisfactory for our main Theorem \ref{t:main_lagr_dim} below. Moreover, since $v$ is a solution to Euler, such a $V^\delta$ solves distributionally 
$$
\partial_t (\nabla V^\delta)=-\nabla\left(\div(v\otimes v)+\nabla p \right)* \rho_\delta
$$
from which, since $p\geq 3$, one deduces that $\partial_t \nabla V^\delta\in L^{\sfrac{p}{2}}_t(L^\infty_x)$ and consequently 
$$
V^\delta\in W^{1,\sfrac{p}{2}}_t(\Lip_x)\subset C^0_t(\Lip_x).
$$
Definition \ref{Lagr_stab_Min_dim}, however, has the desirable feature that it depends only on $v$ and $S$, and does not depend on a choice of mollifying kernel, while the additional stability parameter has the advantage of allowing for weaker or stronger assumptions on $v$.

Definition \ref{Lagr_stab_Min_dim} accords well with the Taylor hypothesis of \emph{frozen turbulence}, which claims that fine structures of a turbulent flow (in this case singularities or more generally points in the dissipation support) are advected by an averaged version of the velocity.  
Indeed, we use Lagrangian neighborhoods and we consider regularizations of the velocity field including mollifications.

In terms of this notion of Lagrangian--type Minkowski dimension we have the following theorem, which represents the precise counterpart of the rough Theorem \ref{t:rough} stated above.

\begin{theorem}[Lagrangian intermittency]
\label{t:main_lagr_dim}
Let $\Omega$ be $\R^d$ or $\T^d$. Let $p\in [3,\infty]$, $\theta\in (0,1)$, $\gamma\in [0,d]$ and $v\in L^p((0,T);B^\theta_{p,\infty}(\Omega))$ be a weak solution of Euler.  Assume that the dissipation set $S\subset \Omega\times (0,T)$ has time stable Lagrangian Minkowski dimension at most $\gamma$, with instability parameters $\beta_1=\theta$ and $\beta_2=1-\theta+\sfrac{(d-\gamma)}{p}$ relative to $v$ (according to Definition \ref{Lagr_stab_Min_dim}). Then $D^v\equiv 0$ if
$$
\theta>\frac{1}{3}-(d-\gamma)\frac{p-3}{3p}.
$$
\end{theorem}
Note that, by reading the previous theorem in terms of structure function exponents (remember $\zeta_p=\theta p$), it gives the linear upper bound \eqref{bound_betamodel} when $d=3$, which is then precisely consistent with
the $\beta-$model, for solutions of Euler that support nontrivial dissipation, modulo of course the notion of dimension we used.  Indeed our theorem implies that, assuming dissipation to be $\gamma$ dimensional in the above sense,  the $\beta-$model prediction is the largest possible choice of $\zeta_p$ which is consistent with anomalous dissipation.  Note also that in case of full codimension, that is $\gamma =0$, the $\beta-$model prediction coincides with the Sobolev embedding $B^{\sfrac{1}{3} }_{3,\infty}\subset L^{\frac{3d}{d-1}-\eps}$: If $\gamma =0$ the  $\beta-$model gives $3\zeta_p= p(1-d)+3d$, which intersects the $p$ axes exactly in $p^*=\frac{3d}{d-1}$.  In other words, assuming $B^{\sfrac{1}{3} }_{3,\infty}$ to be exact in view of the $4/5-$ths law,  velocities with dissipation of full codimension will always have finite $p$ moments, for all $p<p^*$. As for the Eulerian theorem, also Theorem \ref{t:main_lagr_dim} can be generalised to $v\in L^q_t(B^\theta_{p,\infty})$, for any (possibly different) $p,q\geq 3$. See Remark \ref{r:lagrangian_L3_in_time}.  The choice of $\beta_2$ in the above theorem is the one giving the optimal conclusion, but for any given $\beta_2\geq 0$ our proof still gives a nontrivial quantitative relation between $\theta$ and $\gamma$,  with explicit dependence on $\beta_2$. In terms of the Taylor hypothesis, it may be reasonable to assume that $\beta_2 = 0$, since the vector field $V^\delta$, which is purported to carry the singularities of the flow, is allowed to be arbitrarily close to the solution $v$. The dependence of $\beta_2$ on the codimension $d-\gamma$ might, at least heuristically,  be coherent with the following fact: The higher the codimension, larger the portion of space in which the flow is “regular", and thus the more time stability of singularities is plausible.

The proof of the two above Eulerian and Lagrangian type of spatial intermittency theorems will be given in Section \ref{s:spatial_intermitt}. Moreover, by putting together the two theorems from above, one can sort of merge the two notions of dimension and give the following, which can be thought as an intermediate definition between Definition \ref{d:Eul_Min_dim} and Definition \ref{Lagr_stab_Min_dim}.

\begin{definition}[Eulerian--Lagrangian time stability]\label{L_stab_Min_dim}
Let $\Omega$ be $\R^d$ or $\T^d$. Let $S\subset \Omega\times (0,T)$ and $v\in L^q((0,T);L^p(\Omega))$ be an incompressible vector field, for some $p,q\in[1,\infty]$. We will say that $S$ is Eulerian--Lagrangian stable with parameters $\beta_1,\beta_2>0$ relative to $v$ if (there exist implicit constants such that) for all sufficiently small $\delta>0$ there exists a (uniformly in time) Lipschitz incompressible vector field $V^\delta \in C^0([0,T];\Lip(\Omega))$ such that 
\begin{equation}\label{stab_beta1}
\left\|v-V^\delta\right\|_{L^q_t(L^p_x)}\lesssim \delta^{\beta_1},
\end{equation}
and, by letting $\tau=\delta^{\beta_2}$, it holds 
\begin{equation}
\label{L_stab}
\Phi^{V^\delta}_s((S)_{\delta,\tau})\subset (S)_{2\delta,2\tau}
\end{equation}
for every $\delta>0$ small enough and all $|s|< \tau$.
\end{definition}

Condition~\eqref{L_stab} guarantees that the set $S$ does not move too rapidly under the flow map of $V^\delta$.   Under this condition we have that Eulerian time stability (as in Definition~\ref{d:Eul_Min_dim}) implies Lagrangian time stability (as in Definition \ref{Lagr_stab_Min_dim}), and we can apply Theorem~\ref{t:main_lagr_dim} to yield the following corollary.

\begin{corollary}[Eulerian--Lagrangian intermittency]
\label{t:main_lagr}
Let $\Omega$ be $\R^d$ or $\T^d$. Let $p\in [3,\infty]$, $\theta\in (0,1)$, $\gamma\in [0,d]$ and $v\in L^p((0,T);B^\theta_{p,\infty}(\Omega))$ be a weak solution of Euler.  Assume that the dissipation set $S\subset \Omega\times (0,T)$ has Eulerian time stable dimension at most $\gamma$, with instability parameter $\beta=1-\theta+\sfrac{(d-\gamma)}{p}$ (according to Definition \ref{d:Eul_Min_dim}) and moreover it is also Eulerian--Lagrangian stable with parameters $\beta_1=\theta$ and $\beta_2=\beta$ relative to $v$ (according to Definition \ref{L_stab_Min_dim}). Then $D^v\equiv 0$ if
$$
\theta>\frac13-(d-\gamma)\frac{p-3}{3p}.
$$
\end{corollary}
The main point in stating the corollary is to demonstrate that the sharp (in the sense of the $\beta-$model) upper bound \eqref{bound_betamodel} can still be derived if one couples the more standard Eulerian notion of Minkowski dimension from Definition \ref{d:Eul_Min_dim} (in which the neighbourhoods are the most common tubular ones) with the stability requirement \eqref{L_stab}, which basically asks that the dissipation set is not moving too rapidly in time. However, although the corollary is more elementary, in virtue of the Taylor hypothesis one may view Theorem \ref{t:main_lagr_dim} as the main result (that is probably also the most physically relevant) of this paper.

We conclude by the following remark.

\begin{remark}[Vortex Sheets] \label{r:vortexsheets}
When $v$ is a Vortex Sheet it is clear that the instability parameter $\beta$ in Corollary \ref{t:main_lagr} can be taken to be $\beta=1$ (recall that Vortex Sheets belong to $L^\infty_{x,t}\cap L^\infty_t(BV_x)\subset L^\infty_t(B^{\sfrac{1}{p}}_{p,\infty})$ and thus $\theta\geq \sfrac{1}{p}$ for all  values of $p$). Moreover, by choosing $V^\delta=v*\rho_\delta$, the requirement \eqref{L_stab} trivially holds with $\tau\leq \delta$, since in a time interval of length $\sim\tau$ points move at most by $\lesssim \tau$ when subject to a flow of an $L^\infty_{x,t}$ vector field. Thus Corollary \ref{t:main_lagr} also shows that Vortex Sheets lie in the critical (local) energy conservative class for a set that has Eulerian stable Minkowski dimension at most $d-1$ with instability parameter $\beta=1$ (according to Definition \ref{d:Eul_Min_dim}),  or analogously, in virtue of Remark \ref{r:Eul=Minkow}, a space--time set whose standard Minkowski dimension is at most $d$.   We believe this statement to represent a good complement to the energy conservation proved in \cite{Sv09}.
\end{remark}

\subsection{Intermittency in the vanishing viscosity limit}\label{s:interm_vanish_visc}

Here we give the precise statements for vanishing viscosity sequences whose dissipation accumulates on lower--dimensional sets. To this end, for any viscosity $\nu>0$, we let $v^\nu$ to be  a smooth solution to the incompressible Navier-Stokes system \eqref{NS} on $\Omega\times (0,T)$.

Since we wish to deduce properties of the limit $v^\nu\rightarrow v$ (as $\nu\rightarrow 0$) by means of our inviscid intermittency results stated in the section above, we need to rephrase the anomalous dissipation definition \eqref{0th_law} as an interior dissipation condition
\begin{equation}
    \label{0th_law_new}
       \liminf_{\nu\rightarrow 0}\nu\int_\delta^{T-\delta}\int_{B_R}|\nabla v^\nu|^2\,dxdt>0, \qquad \text{for some } \delta,R>0.
\end{equation}
This latter technical reformulation of the $0-$th law ensures that not all the dissipation is being pushed towards the time boundary\footnote{Note that in the recent constructions \cite{CCS22,BD22} the dissipation, in the limit, pushes all the mass towards $T$.}, or is escaping to infinity when $\Omega=\R^d$, which seems to be a physically reasonable assumption. The use of this assumption in our context is that both Theorems \ref{t:main} and \ref{t:main_lagr_dim} conclude intermittency by means of $D^v\not\equiv 0$, and $D^v$, being a distribution, does not see masses at boundaries or at infinity.  Note that, when $\Omega=\T^d$, condition \eqref{0th_law_new} is equivalent to the usual 0--th law \eqref{0th_law} as soon as $\{v^\nu\}_{\nu>0}$ is a compact  $C^0_t (L_x^2)$ sequence (as can be seen from \eqref{leray_energy_equality} below),
and compactness in $C^0_t (L_x^2)$ itself follows from a uniform bound in $L_t^\infty (B_{2,\infty}^\theta)$, for some $\theta >0$, by the Aubin-Lions-Simon Lemma.

In what follows we denote by $\mathcal{M}^+(\Omega\times [0,T])$ the space of finite positive space--times measures, that is the dual of $C_b^0(\Omega\times [0,T])$, where the pedex $b$ stands for \emph{bounded}. Thus we will adopt the usual notion of convergence in measure by duality with continuous functions. We start by stating the first result, which uses the Eulerian time stable Minkowski dimension of Definition \ref{d:Eul_Min_dim} with $\beta=1$, which then is equivalent to the standard space--time Minkowski dimension $\gamma+1$ in virtue of Remark \ref{r:Eul=Minkow}.

\begin{theorem}[Eulerian vanishing viscosity intermittency]\label{t:eul_vanish_visc}
    Let $\Omega$ be $\T^d$ or $\R^d$. Let $\{v^\nu\}_{\nu>0}$ be a sequence of smooth solutions to \eqref{NS}  enjoying \eqref{0th_law_new}, such that  $\nu |\nabla v^\nu|^2\rightharpoonup\mu$ in $\mathcal{M}^+(\Omega\times [0,T])$. Suppose that the set $S:= \spt_{x,t} \mu$ it has space--time Minkowski dimension at most $\gamma+1$. Then, for $p\in[3,\infty]$, by letting $\theta\in(0,1)$ be any positive number such that 
$$
\frac{2\theta}{1-\theta}>1-\frac{p-3}{p}(d-\gamma),
$$
it must hold that
\begin{equation}
    \label{inter_norm_blowup_euler}
\lim_{\nu\rightarrow 0} \norm{v^{\nu}}_{L^p_t(B^{\theta}_{p,\infty})}=+\infty.
\end{equation}
\end{theorem}

The assumption $\nu |\nabla v^\nu|^2\rightharpoonup \mu$ as measures is not restrictive since, as soon as the initial data $v_0^\nu$ belong to a compact subset of $L^2(\Omega)$, then the global energy balance 
\begin{equation}\label{leray_energy_equality}
\frac12 \int_{\Omega}|v^\nu(x,t)|^2\,dx+\nu \int_0^t \int_{\Omega}|\nabla v^\nu(x,s)|^2\,dxds=\frac12 \int_{\Omega}|v_0^\nu(x)|^2\,dx,\quad \forall t\geq 0,
\end{equation}
 implies that the sequence $\{\nu|\nabla v^\nu|^2\}_{\nu>0}$ is bounded in $L^1(\Omega\times (0,T))$, and thus it converges to a non-negative and finite measure $\mu$, up to taking a subsequence. Thus the only true assumption in the previous theorem is the existence of a lower--dimensional accumulation point (in the space of non-negative finite measures) in the sequence of dissipations.
The previous global energy equality is obtained by integrating in space, and then in time, its (stronger) local version 
$$
\partial_t\left( \frac{|v^\nu|^2}{2}\right)+\div \left(\left(\frac{|v^\nu|^2}{2}+p^\nu \right) v^\nu \right)-\nu\Delta\left(\frac{|v^\nu|^2}{2}\right)+\nu |\nabla v^\nu|^2=0,
$$
which, for smooth solutions $v^\nu$, is obtained as usual by scalar multiply the first equation in \eqref{NS} by $v^\nu$ itself and using Leibniz rule, together with the incompressibility $\div v^\nu=0$. We emphasize that the assumption \eqref{0th_law_new} implies that at least a portion of the support of $\mu$ is strictly inside $\Omega\times(0,T)$.

The Lagrangian version of the previous theorem, which is indeed the rigorous counterpart of Theorem \ref{t:rough_visc}, reads as follows. 

\begin{theorem}[Lagrangian vanishing viscosity intermittency]\label{t:lagr_vanish_visc}
    Let $\Omega$ be $\T^d$ or $\R^d$. Let $\{v^\nu\}_{\nu>0}$ be a sequence of smooth solutions to \eqref{NS}  enjoying \eqref{0th_law_new}, such that  $v^\nu \rightarrow v$ in $L^2(\Omega\times (0,T))$ and $\nu |\nabla v^\nu|^2\rightharpoonup \mu$ in $\mathcal{M}^+(\Omega\times [0,T])$ in the inviscid limit. Let $\theta\in(0,1)$, $p\in[3,\infty]$ and assume $v\in L^p(\Omega\times (0,T))$ and that the set $S:=\spt_{x,t} \mu$ has time stable Lagrangian Minkowski dimension at most $\gamma$, with instability parameters $\beta_1=\theta$ and $\beta_2=1-\theta+\sfrac{(d-\gamma)}{p}$ relative to $v$ (according to Definition \ref{Lagr_stab_Min_dim}). Then, if $
\theta>\sfrac13-(d-\gamma)\sfrac{(p-3)}{3p}$,
it must hold that
\begin{equation}
    \label{inter_norm_blowup_lagr}
\lim_{\nu\rightarrow 0} \norm{v^{\nu}}_{L^p_t(B^{\theta}_{p,\infty})}=+\infty.
\end{equation}
\end{theorem}

Note that, differently from Theorem \ref{t:eul_vanish_visc}, in Theorem \ref{t:lagr_vanish_visc} we had to assume $L^2_{x,t}$ compactness of the sequence $\{v^\nu\}_{\nu>0}$ and knowledge of a limiting velocity field.
This is indeed a drawback of the Lagrangian Minkowski notion of dimension, while the Eulerian Minkowski one,  being purely geometrical, has the nice feature  that it does not require any reference vector field in order to be defined.  However, in principle, if one knows the rate of convergence in Theorem~\ref{t:lagr_vanish_visc}, then either $v^\nu$ itself or regularizations of the $v^\nu$ provide a good candidate for the family of approximations $V^\delta$ required in Definition \ref{Lagr_stab_Min_dim}.  Similarly to Theorem \ref{t:main_lagr_dim}, the use of the Lagrangian Minkowski dimension in the previous theorem can be seen as a theoretical instance of the Taylor's hypothesis for which coarse structures of large Reynolds number incompressible flows follow an averaged velocity field.
The proofs of Theorem \ref{t:eul_vanish_visc} and Theorem \ref{t:lagr_vanish_visc} will be given in Section \ref{s:vanish_visc_proof}. Let us point out again that, for the interest of possible empirical testing, the precise assumptions in Theorem~\ref{t:lagr_vanish_visc} are designed to optimize the exponent, and that weaker conclusions can be obtained if weaker assumptions are satisfied.

\section{Notations and main tools}\label{Sec:tools}

Here we list some useful tools that will be used throughout the paper.  

\subsection{Minkowski dimension}
We start by recalling the basics on the notion of Minkowski dimension, while for a more detailed account we refer to \cite[Chapter 5]{Mat}. For any bounded set $S\subset \R^d$, $d \geq 1$, we define 
\begin{equation}\label{def:Minkowskiunif2}
\overline{\dim}_{\mathcal{M}} S := \inf \left\{ s \geq 0: \, \overline{\mathcal{M}}^{s} \left( S \right)=0  \right\} ,
\end{equation}
where $\overline{\mathcal{M}}^{ s}$ denotes the upper Minkowski content, obtained by computing the  shrinking rate (in volume) of the $\delta-$neighbourhoods $\left(S\right)_{\delta}:=\{ x \in \R^d: \, \dist(x,S)\leq \delta \}.$ More explicitly,  
\begin{equation}\label{def:Minkowskiunif}
\overline{\mathcal{M}}^{ s}\left( S \right):= \limsup_{\delta \rightarrow 0^+}   \frac{\mathcal{H}^d\left( \left(S\right)_{\delta}\right) }{\delta^{d-s}}.
\end{equation}
Clearly $\overline \dim_\mathcal{M}S\in [0,d]$. The boundedness of the set $S$ is necessary in \eqref{def:Minkowskiunif} since otherwise one could have $\mathcal{H}^d((S)_\delta)=+\infty$ for every $\delta>0$. When $S$ is unbounded, one can define the \emph{local} version of \eqref{def:Minkowskiunif2} with the usual straightforward modifications. Note that by \eqref{def:Minkowskiunif}, letting $\gamma=\overline \dim_{\mathcal{M}} S$, then if $\tilde \gamma>\gamma$ one has the asymptotic
\begin{equation}
    \label{Mink_asympt}
    \mathcal{H}^d((S)_\delta)\lesssim \delta^{d-\tilde \gamma},
\end{equation}
for all $0<\delta\ll 1$ sufficiently small and for some implicit constant which might depend on $d,\gamma,\tilde \gamma$, but it is otherwise independent on $\delta$. For simplicity, in view of \eqref{Mink_asympt} where $\tilde \gamma$ can be chosen arbitrarily close to $\gamma$, we will say that $S$ has upper Minkowski dimension at most $\gamma$ if $\mathcal{H}^d((S)_\delta)\lesssim \delta^{d-\gamma}$, thus slightly abusing terminology. In particular, that terminology has been used in Remark \ref{r:Eul=Minkow}.
\\
\\
It readily follows from the definition that this gives a (strictly) stronger notion of dimension with respect to the Hausdorff one,  that is $\underline{\dim}_{\mathcal M} S\geq \dim_{\mathcal H} S$ for any set $S$. Indeed, every dense set automatically has full upper Minkowski dimension (i.e. $\overline \dim_\mathcal{M} S=d$), while the Hausdorff one could even be $0$ (see \cite{Mat} for examples).  Thus, in terms of dimensional lower bounds and in view of possible accumulation of energy dissipation on dense sets, it might be desirable to relax the notion of dimension used in this work to Hausdorff. We refer to \cite{DDI23} for a recent result relating $\dim_{\mathcal H}$ of the dissipation set to the $L^p$ integrability of quite general fluid models by means of measure theoretic tools,  thus quite different than the approach we used here. An equivalent way of formulating the upper Minkowski dimension is in terms of \emph{counting boxes}: At most $\delta^{-\gamma}$ cubes of size $\delta$ are needed to cover the set $S$, in the limit as $\delta\rightarrow 0$ (see for instance \cite{Mat}*{Chapter 5}).

Notice that both the Minkowski--type notions of dimension we introduced in Definitions \ref{d:Eul_Min_dim} and \ref{Lagr_stab_Min_dim} could have been rephrased more rigorously in terms of an appropriate  upper Minkowski content, but since this would have not added anything substantial to the results, we avoid this technicality and directly define the dimension by looking at the shrinking rates \eqref{E_Min_dim_gamma_def} and \eqref{L_dim_gamma}. 

\subsection{Besov spaces} Let $\Omega$ be the whole space $\R^d$ or the $d-$dimensional torus $ \T^d$. For every $p\in [1,\infty]$, $k\in \N$ we will denote by $L^p(\Omega)$ the usual Lebesgue space of measurable $p-$integrable functions and by $W^{k,p}(\Omega)$ the Sobolev space of functions whose $k-$th order derivatives belong to $L^p(\Omega)$.

Moreover recall the definition of the Besov space for any $\theta\in(0,1)$
$$
B^\theta_{p,\infty}(\Omega)=\left\{f\in L^p(\Omega)\,:\, [f]_{B^\theta_{p,\infty}(\Omega)}<\infty \right\},
$$
where
$$
[f]_{B^\theta_{p,\infty}(\Omega)}=\sup_{h\neq 0} \frac{\|f(\cdot+h)-f(\cdot)\|_{L^p(\Omega)}}{|h|^\theta}.
$$
Then the full Besov norm will be given by $\|f\|_{B^\theta_{p,\infty}(\Omega)}=\|f\|_{L^p(\Omega)}+[f]_{B^\theta_{p,\infty}(\Omega)}$.

\subsection{Mollification estimates and commutators}
For every function $f:\Omega\rightarrow \R$ and $\eps>0$ we will write $f_\eps=f*\rho_\eps$ to denote the regularisation of $f$, where $\rho_\eps$ is a standard Friedrichs' mollifier.  We start by recalling the following standard mollification estimates
\begin{align}
\|f-f_\eps\|_{L^p}&\lesssim \eps^\theta [f]_{B^\theta_{p,\infty}}\label{moll1}\\
\|f_\eps\|_{W^{k+1,p}}&\lesssim \eps^{\theta-k-1}[f]_{B^\theta_{p,\infty}}\label{moll2}
\end{align}
for every $k\in \N_0$, $p\in [1,\infty]$ and $\theta\in (0,1)$.

Moreover, we also have the following Constantin-E-Titi type commutator estimate
\begin{equation}\label{CET_commutator}
\|(f g)_\eps-f_\eps g_\eps\|_{W^{k,p}}\lesssim \eps^{\theta+\beta-k} [f]_{B^\theta_{r p,\infty}} [g]_{B^\beta_{r' p,\infty}},
\end{equation}
for every $k\in \N_0$, $p,r\in [1,\infty]$, $\theta,\beta\in (0,1)$, where $\frac{1}{r}+\frac{1}{r'}=1$ and $rp, r'p \geq 1$. 
The latter quadratic commutator estimate (in its particular case $k=0$) is the key observation which led to the proof of the positive part of the Onsager's conjecture in \cite{CET94}, namely the \emph{total} kinetic energy conservation for solutions $v\in L^3_t(B^{\sfrac13 +}_{3,\infty})\cap C^0_t(L^2_x)$.

Indeed, letting $R_\eps:=v_\eps \otimes v_\eps- (v\otimes v)_\eps$, the spatial regularisation $v_\eps:=v*\rho_\eps$ solves the Euler-Reynolds system
$$
\partial_t v_\eps +\div (v_\eps\otimes v_\eps)+\nabla p_\eps=\div R_\eps,
$$
from which, by scalar multiplying by $v_\eps$ itself and using the incompressiblity $\div v_\eps=0$, we get the local (approximate) energy balance
\begin{equation}\label{local_energy_eps}
    \partial_t \left(\frac{|v_\eps|^2}{2}\right)+\div \left(\left(\frac{|v_\eps|^2}{2}+p_\eps \right) v_\eps \right) =v_\eps \cdot \div R_\eps=\div (R_\eps v_\eps)-R_\eps :\nabla v_\eps.
\end{equation}
Then \eqref{moll2} and \eqref{CET_commutator} imply the estimate
\begin{equation}\label{estimate_DRmeasure_approx}
    \|R_\eps :\nabla v_\eps\|_{L^{\sfrac{p}{3}}_{x,t}}\leq \| \nabla v_\eps\|_{L^{p}_{x,t}} \|R_\eps\|_{L^{\sfrac{p}{2}}_{x,t}}\lesssim \eps^{3\theta-1}\| v\|^3_{L^p_t(B^\theta_{p,\infty})},
\end{equation}
whenever $v\in L^p_t(B^\theta_{p,\infty})$, for $p\geq 3$ and $\theta\in (0,1)$, which in particular proves the local (and thus global) energy conservation if $\theta>\sfrac{1}{3}$.

Since in our case we are interested in the local energy balance \eqref{Local_energy}, we will also need to handle a third order commutator that naturally arises when averaging the left hand side of \eqref{Local_energy}, see Lemma \ref{l:higher_average} below.  For this purpose we define 
$$
K^{f,g}_\eps(x):=\int_{\Omega} |f(x-h)-f_\eps(x)|^2(g(x-h)-g_\eps(x))\rho_\eps(h)\,dh.
$$
By its trilinear structure, the previous commutator will enjoy a suitable cubic estimate for which we give a proof for the reader's convenience.  
The purely cubic commutator that arises here is a special instance of a {\it third order cumulant}; see \cite{E07} for a discussion of general cumulant expansions.

\begin{proposition}\label{p:cubic_comm}
Let $\Omega=\R^d,\T^d$. For every $k\in \N_0$, $\theta\in (0,1)$, $p\geq 3$ and any regular enough $f, g :\Omega\rightarrow \R$, we have
\begin{equation}\label{e:cubic_comm}
\|K^{f,g}_\eps\|_{W^{k,\sfrac{p}{3}}}\lesssim \eps^{3\theta-k} [f]^2_{B^\theta_{p,\infty}}[g]_{B^\theta_{p,\infty}}.
\end{equation}
\end{proposition}
\begin{proof}
We will only give the proof for $k=0,1$, since then the case $k\geq 2$ can be done by iterating the very same computations.  Actually, in this manuscript we will only make use of the estimate \eqref{e:cubic_comm} when $k=1$. 

Since $\rho_\eps\,dh$ is a probability measure on $\Omega$,  we have
$$
|K^{f,g}_\eps (x)|^{\frac{p}{3}}\leq \int_{\Omega} |f(x-h)-f_\eps(x)|^{\frac{2p}{3}}|g(x-h)-g_\eps(x)|^{\frac{p}{3}}\rho_\eps(h)\,dh.
$$
Thus by H\"older inequality 
\begin{equation}
\label{holder1}
\|K^{f,g}_\eps\|_{L^{\sfrac{p}{3}}}^{\frac{p}{3}}\leq \int_{\Omega}\|f(\cdot-h)-f_\eps(\cdot)\|_{L^p}^{\frac{2p}{3}}\|g(\cdot-h)-g_\eps(\cdot)\|_{L^p}^{\frac{p}{3}}\rho_\eps(h)\,dh.
\end{equation}
By using \eqref{moll1}, together with the definition of the Besov norms, we can estimate
\begin{equation}
\label{Besov_split}
\|f(\cdot-h)-f_\eps(\cdot)\|_{L^p}\leq \|f(\cdot-h)-f(\cdot)\|_{L^p}+ \|f-f_\eps\|_{L^p}\lesssim \left(|h|^\theta+\eps^\theta\right)[f]_{B^\theta_{p,\infty}}
\end{equation}
Thus \eqref{holder1} gets to
\begin{equation}
\label{est_order0}
\|K^{f,g}_\eps\|_{L^{\sfrac{p}{3}}}^{\frac{p}{3}}\lesssim [f]_{B^\theta_{p,\infty}}^{\frac{2p}{3}}[g]^{\frac{p}{3}}_{B^\theta_{p,\infty}}\int_{\Omega}\left(|h|^\theta+\eps^\theta\right)^p\rho_\eps(h)\,dh\lesssim \eps^{\theta p}[f]_{B^\theta_{p,\infty}}^{\frac{2p}{3}}[g]^{\frac{p}{3}}_{B^\theta_{p,\infty}},
\end{equation}
where in the last inequality we have also used that the kernel $\rho_\eps$ is supported in the ball $B_\eps(0)$.  
It is clear that \eqref{est_order0} gives \eqref{e:cubic_comm} for $k=0$. 

In the case $k=1$, we need to estimate a derivative of $K^{f,g}_\eps$. To do so, let $\partial K^{f,g}_\eps$ be any partial derivative of the cubic commutator. It is clear that the derivative can hit $f_\eps,g_\eps$ or $\rho_\eps$. More precisely we will need to estimate the following three kind of terms
\begin{align*}
K_1(x)&=\int_{\Omega} \partial f_\eps(x) (f(x-h)-f_\eps(x))(g(x-h)-g_\eps(x))\rho_\eps(h)\,dh\\
K_2(x)&=\int_{\Omega}  |f(x-h)-f_\eps(x)|^2 \partial g_\eps(x)\rho_\eps(h)\,dh\\
K_3(x)&=\int_{\Omega} |f(x-h)- f_\eps(x)|^2(g(x-h)-g_\eps(x))\partial \rho_\eps(h)\,dh.
\end{align*}
By using \eqref{moll2} and \eqref{Besov_split} we can bound
\begin{align}\label{e:K1}
\|K_1\|_{L^{\sfrac{p}{3}}}^{\frac{p}{3}}&\leq  \|\partial f_\eps\|_{L^p}^{\frac{p}{3}}\int_{\Omega}\|f(\cdot-h)-f_\eps(\cdot)\|_{L^p}^{\frac{p}{3}}\|g(\cdot-h)-g_\eps(\cdot)|_{L^p}^{\frac{p}{3}}\rho_\eps(h)\,dh\nonumber\\
&\lesssim \eps^{(\theta-1)\frac{p}{3}} [f]_{B^\theta_{p,\infty}}^{\frac{2p}{3}}[g]_{B^\theta_{p,\infty}}^{\frac{p}{3}}\int_{\Omega}\left(|h|^\theta+\eps^\theta\right)^p\rho_\eps(h)\,dh\nonumber\\
&\lesssim  \eps^{(3\theta-1)\frac{p}{3}} [f]_{B^\theta_{p,\infty}}^{\frac{2p}{3}}[g]_{B^\theta_{p,\infty}}^{\frac{p}{3}}.
\end{align}
Similarly we also get 
\begin{equation}
\label{e:K2}
\|K_2\|_{L^{\sfrac{p}{3}}}^{\frac{p}{3}}\lesssim \eps^{(3\theta-1)\frac{p}{3}} [f]_{B^\theta_{p,\infty}}^{\frac{2p}{3}}[g]_{B^\theta_{p,\infty}}^{\frac{p}{3}}.
\end{equation}
To estimate $K_3$ we notice that a derivative on the kernel $\rho_\eps$ will only bring and $\eps^{-1}$ more to the estimate \eqref{est_order0} that we already proved. More precisely, since $\eps \partial \rho_\eps$ has bounded $L^1$ norm, we have
$$
|K_3 (x)|^{\frac{p}{3}}\lesssim \eps^{1-\frac{p}{3} }\int_{\Omega} |f(x-h)-f_\eps(x)|^{\frac{2p}{3}}|g(x-h)-g_\eps(x)|^{\frac{p}{3}}|\partial \rho_\eps(h)|\,dh,
$$
from which 
\begin{equation}
\label{e:K3}
\|K_3\|_{L^{\sfrac{p}{3}}}^{\frac{p}{3}}\lesssim \eps^{1-\frac{p}{3} } [f]_{B^\theta_{p,\infty}}^{\frac{2p}{3}}[g]^{\frac{p}{3}}_{B^\theta_{p,\infty}}\int_{\Omega}\left(|h|^\theta+\eps^\theta\right)^p|\partial \rho_\eps(h)|\,dh\lesssim \eps^{\theta p-\frac{p}{3}}[f]_{B^\theta_{p,\infty}}^{\frac{2p}{3}}[g]^{\frac{p}{3}}_{B^\theta_{p,\infty}},
\end{equation}
where to obtain the last inequality we have also used $\int |\partial \rho_\eps|\lesssim \eps^{-1}$. 
Inequalities \eqref{e:K1}, \eqref{e:K2} and \eqref{e:K3} give \eqref{e:cubic_comm} for $k=1$.
\end{proof}

\subsection{Double pressure regularity}
We conclude this section by recalling the double regularity of the pressure from \cite{CDF20}
\begin{equation}\label{p_double_reg}
\|p(t)\|_{B^{2\theta}_{\sfrac{p}{2},\infty}}\lesssim \|v(t)\|^2_{B^{\theta}_{p,\infty}},
\end{equation}
for every $\theta\in (0,\sfrac{1}{2})$, $p\in (2,\infty)$. The case $p=\infty$ can be found in \cite{Isett17,CD18}. The double regularity estimate \eqref{p_double_reg} will play a crucial role in the proof of Theorem \ref{t:main} and Theorem \ref{t:main_lagr_dim} since, together with \eqref{CET_commutator},  will allow us to estimate the pressure commutator $P_\eps =(pv)_\eps-p_\eps v_\eps$ and deduce that it also behaves cubically in $\theta$.

\section{Intermittency from lower dimensional time dissipation}\label{s:time_intermitt}
In this section we prove the results concerning lower dimensional singularities in time as claimed in Section \ref{Intro:time}.

\subsection{Proof of Proposition  \ref{p:en_besov_regular}}
The proof is very similar to the case $p=\infty$ proved in \cite{Is2013,CD18}, but here we take the integral in time instead of the supremum and apply Young's inequality.  

Thus assume $p<\infty$ and fix $h> 0$. Mollify $v$ in space at scale $\eps>0$ to obtain the smooth velocity $v_\eps$. We will chose the regularization length scale $\eps=\eps (h)$ at the very end in order to balance the terms. Split
\begin{align*}
\|e_v(\cdot+h)-e_v(\cdot)\|_{L^{\sfrac{p}{3}}_t}&\leq \|e_v(\cdot+h)-e_{v_\eps}(\cdot+h)\|_{L^{\sfrac{p}{3}}_t}+\|e_{v_\eps}(\cdot+h)-e_{v_\eps}(\cdot)\|_{L^{\sfrac{p}{3}}_t}\\
&\quad+ \|e_{v_\eps}(\cdot)-e_{v}(\cdot)\|_{L^{\sfrac{p}{3}}_t}=(I)+(II)+(III).
\end{align*}
The two terms $(I),(III)$ enjoy a very similar estimate. Indeed, since the mollification is average preserving,  by using \eqref{CET_commutator} for every fixed time slice, we have
\begin{align*}
\left| \int_{\Omega} |v(x,t)|^2-|v_\eps(x,t)|^2\,dx\right|&=\left| \int_{\Omega} |v(x,t)|^2_\eps-|v_\eps(x,t)|^2\,dx\right|\lesssim \eps^{2\beta}[v(t)]^2_{B^\beta_{2,\infty}},
\end{align*}
from which, by also integrating in time on the interval $(0,T)$, we get
\begin{equation}\label{increment1-3}
(I),(III)\lesssim \eps^{2\beta}[v]^2_{L^{\sfrac{2p}{3}}_t \left(B^\beta_{2,\infty}\right)}.
\end{equation}
Moreover, by integrating on $\Omega$ the Euler-Reynolds local energy balance \eqref{local_energy_eps} we get
$$
\frac{d}{dt} e_{v_\eps}=\frac{1}{2}\frac{d}{dt}\left(\int_\Omega |v_\eps|^2\,dx\right)=-\int_\Omega  R_\eps : \nabla v_{\eps} \,dx. 
$$

Thus to estimate $(II)$ we use the usual Constantin-E-Titi quadratic estimate \eqref{CET_commutator} together with \eqref{moll2}
\begin{align*}
\left| e_{v_\eps}(t+h)-e_{v_\eps}(t)\right|&\leq \int_t^{t+h}\left| e'_{v_\eps}(s) \right|\,ds\leq \int_t^{t+h}\int_{\Omega}\left| R_\eps : \nabla v_{\eps} \right|\,dxds\\
&\lesssim \eps^{3\theta-1}\int_0^T[v(s)]^3_{B^\theta_{3,\infty}}\mathbbm{1}_{(t,t+h)}(s)\,ds.
\end{align*}
Thus by Young's inequality
\begin{align}\label{increment2}
\|e_{v_\eps}(\cdot+h)-e_{v_\eps}(\cdot)\|_{L^{\sfrac{p}{3}}_t}&\lesssim  \eps^{(3\theta-1)}\left(\int_0^T[v(s)]^{p}_{B^\theta_{3,\infty}} ds \right)^{3/p} \| 1_{(-h, 0)}(s) \|_{L^1} \notag \\
&\lesssim \eps^{(3\theta-1)} h [v]_{L^{p}_t\left(B^\theta_{3,\infty} \right)}^3.
\end{align}

Collecting \eqref{increment1-3} and \eqref{increment2} we obtain
$$
\|e_v(\cdot+h)-e_v(\cdot)\|_{L^{\sfrac{p}{3}}_t}\lesssim \eps^{2\beta}+\eps^{3\theta-1}h,
$$
from which by choosing $\eps =h^{\frac{1}{1-3\theta+2\beta}}$ we conclude the proof.  
\qed

\subsection{Proof of Theorem \ref{t:en_cons_in_time}}
Fix any $\eps>0$. By Proposition \ref{p:en_besov_regular} we have 
$$
e_v\in B^{\frac{2\beta}{1-3\theta+2\beta}}_{{\sfrac{p}{3}},\infty}([0,T])\subset W^{\frac{2\beta}{1-3\theta+2\beta}-\eps,{\sfrac{p}{3}}}([0,T]),
$$ 
from which
$$
e'_v\in W^{\frac{2\beta}{1-3\theta+2\beta}-\eps-1,{\sfrac{p}{3}}}([0,T])=\left( W^{1+\eps-\frac{2\beta}{1-3\theta+2\beta},({\sfrac{p}{3}})'}([0,T])\right)^*,
$$
being $({\sfrac{p}{3}})'=\frac{p}{p-3}$ the H\"older conjugate of ${\sfrac{p}{3}}$ and $X^*$ the usual dual vector space of $X$.

Let $\delta>0$, $(S_T)_\delta$ a $\delta-$neighbourhood of the set $S_T$ and $\eta_\delta\in C^\infty([0,T])$ be such that
$$
\eta_{\delta}\big|_{(S_T)_\delta}\equiv 1,\quad  \eta_{\delta}\big|_{[0,T]\setminus (S_T)_{2\delta}}\equiv 0.
$$
Then for every $\alpha\in (0,1)$, $q\in (1,\infty)$, $\tilde \gamma>\gamma$, we 
can estimate by \eqref{Mink_asympt}
\begin{equation}\label{est_on_eta_time}
\| \eta_\delta\|_{W^{\alpha,q}}\leq \| \eta_\delta\|^{1-\alpha}_{L^q}\| \eta'_\delta \|^\alpha_{L^q}\lesssim \delta^{(1-\tilde \gamma)\frac{1}{q}-\alpha}.
\end{equation}
Let $\varphi\in C_c^\infty((0,T))$. 
Since 
the distribution $e'_v$ is supported on $S_T$, we have 
\begin{equation}
\label{est_funct_e'}
|\langle e'_v,\varphi\rangle|=|\langle e'_v,\eta_\delta \varphi\rangle|\lesssim \| \eta_\delta \varphi\|_{W^{1+\eps-\frac{2\beta}{1-3\theta+2\beta},({\sfrac{p}{3}})'}}\lesssim \delta^{(1-\tilde \gamma)\frac{p-3}{p}+\frac{2\beta}{1-3\theta+2\beta}-1-\eps},
\end{equation}
where in the last inequality we used \eqref{est_on_eta_time}. Thus if
$$
\frac{2\beta}{1-3\theta+2\beta}>1-\frac{p-3}{p}(1-\gamma).
$$
we can chose $\eps>0$ sufficiently small and $\tilde \gamma$ sufficiently close to $\gamma$, such that the exponent of $\delta$ in the right hand side of \eqref{est_funct_e'} is strictly positive. The proof is concluded by letting $\delta\rightarrow 0$.
\qed

\section{Intermittency from lower dimensional spatial dissipation}\label{s:spatial_intermitt}
In this section we will prove our main intermittency--type results when the dissipation support is not space filling. To avoid tedious technicalities that arise when considering the whole space $\R^d$ (for instance the need to introduce the bounded set $M$ in the definition of the Minkowski dimensions in \eqref{E_Min_dim_gamma_def} and \eqref{L_dim_gamma}), we will give the full proofs when $\Omega=\T^d$. Then the case $\Omega=\R^d$ can be reconstructed by straightforward modifications. The general strategy in the proofs of Theorem \ref{t:main} and Theorem \ref{t:main_lagr_dim} is the same, but since the two differs in some important technical details (like the definition of the cut--off function and the way the transport error is handled) we have provided the full details for each of them separately for the convenience of the reader at the risk of some redundancy. 

Before proving our main theorems we start with the following lemma which  plays a crucial role in the averaging process of the local energy balance \eqref{Local_energy}, with emphasis on the cubic term $|v|^2 v$. 
\begin{lemma}[Higher--order averaging]
\label{l:higher_average}
By denoting $f_\eps$ to be the spatial mollification of $f$ at scale $\eps$, we have
$$
\left(|v|^2 v \right)_\eps=K^v_\eps -2R_\eps v_\eps -v_\eps \tr R_\eps + |v_\eps|^2 v_\eps,
$$
where $R_\eps=v_\eps\otimes v_\eps -(v\otimes v)_\eps$ is the quadratic Constantin-E-Titi commutator, $R_\eps v_\eps$ is the usual notation for the $(d\times d)-$matrix $R_\eps$ applied to the vector $v_\eps$ and 
$$
K^v_\eps(x):=\int |v(x-h)-v_\eps(x)|^2 (v(x-h)-v_\eps(x))\rho_\eps(h)\,dh.
$$
\end{lemma}

\begin{proof}
Fix $x\in \Omega$.  By thinking of the mollification $(\cdot )_\eps$ as an average with respect to the probability measure $\rho_\eps(x-y)\,dy$ and $v_\eps(x)$ as a constant,  we have 
$$
\big( v-v_\eps(x)\big)_\eps =0.
$$
Thus we can expand
\begin{align*}
\left(|v|^2 v \right)_\eps&=\left(|v|^2 (v-v_\eps(x)) \right)_\eps+|v|^2_\eps v_\eps=\left(|v-v_\eps(x)+v_\eps(x)|^2 (v-v_\eps(x)) \right)_\eps+|v|^2_\eps v_\eps\\
&=\left(|v-v_\eps(x)|^2 (v-v_\eps(x)) \right)_\eps+|v_\eps(x)|^2 \big(v-v_\eps(x) \big)_\eps\\
&\quad+2  \big((v-v_\eps(x))\cdot  v_\eps(x) (v-v_\eps(x)) \big)_\eps+|v|^2_\eps v_\eps\,.
\end{align*}
Moreover 
\begin{align*}
 \big((v-v_\eps(x))\cdot  v_\eps(x) (v-v_\eps(x)) \big)_\eps&= \big(v\cdot  v_\eps(x) (v-v_\eps(x)) \big)_\eps-|v_\eps(x)|^2\big( v-v_\eps(x)\big)_\eps\\
 &=\big( v\cdot v_\eps(x) v\big)_\eps-|v_\eps|^2v_\eps=-R_\eps v_\eps\,,
\end{align*}
from which we get 
\begin{align*}
\left(|v|^2 v \right)_\eps&=\left(|v-v_\eps(x)|^2 (v-v_\eps(x)) \right)_\eps-2R_\eps v_\eps+|v|^2_\eps v_\eps\\
&=\left(|v-v_\eps(x)|^2 (v-v_\eps(x)) \right)_\eps-2R_\eps v_\eps+\left(|v|^2_\eps - |v_\eps|^2\right)  v_\eps +  |v_\eps|^2  v_\eps\,.
\end{align*}
\end{proof}

The following lemma is standard
\begin{lemma}\label{l:eulerian_cutoff}
Define
 $
 (S)_{\delta,\delta}=\left\{(x+B_\delta(0),t+B_\delta(0))\, :\, (x,t)\in S \right\}
 $
 to be the space--time $\delta,\delta-$neighbourhood of the set $S$ which enjoys \eqref{E_Min_dim_gamma} and let $q\geq 1$. There exists $\chi_\delta\in C^\infty([0,T]\times \T^d)$ such that $ \chi_\delta\big|_{(S)_{\delta,\delta}}\equiv 1$,  $ \chi_\delta\big|_{\left((S)_{4\delta,4\delta}\right)^c}\equiv 0$ and 
\begin{align}
\|\chi_\delta\|_{L^q_{x,t}}&\lesssim \delta^{(d-\gamma)\frac{1}{q}}\label{est_chi}\\
\| \partial_t \chi_\delta\|_{L^q_{x,t}}+\| \nabla \chi_\delta\|_{L^q_{x,t}}&\lesssim \delta^{(d-\gamma)\frac{1}{q}-1}.\label{est_chi_der}
\end{align}
\end{lemma}
\begin{proof}
Let $\mathbbm{1}_{(S)_{2\delta,2\delta}}(x,t)$ be the indicator function of the set $(S)_{2\delta,2\delta}$. Then the space--time mollification at scale $\delta$, $\chi_\delta=\mathbbm{1}_{(S)_{2\delta,2\delta}}*_{x,t} \rho_\delta$ satisfies all the claimed properties, where 
$$
\rho_\delta(x,t):=\frac{1}{\delta^{d+1}}\rho\left(\frac{x}{\delta},\frac{t}{\delta}\right)
$$
for some $\rho\in C^\infty_c$, with  $\int_{\mathbb R}\int_{{\mathbb R^d}} \rho\,dxdt=1$.
\end{proof}

\subsection{Proof of Theorem \ref{t:main}}
Let $D_\eps^v=R_\eps:\nabla v_\eps$ be the approximation of the Duchon-Robert distribution, as in \eqref{DR_measure}. Let $\varphi\in C^\infty_c(\T^d\times (0,T))$ be a test function.  We want to show that $\langle D_\eps^v,\varphi\rangle\rightarrow 0$ as $\eps \rightarrow 0$.  Let $\chi_\delta$ be the cut--off function given by Lemma \ref{l:eulerian_cutoff}. We split the action of $D^v$ as 
\begin{equation}\label{split_D}
\langle D_\eps^v,\varphi\rangle=\langle D_\eps^v,\varphi\chi_\delta\rangle+\langle D_\eps^v,\varphi(1-\chi_\delta)\rangle.
\end{equation}
The first term is easy to estimate by using \eqref{est_chi} together with \eqref{estimate_DRmeasure_approx}
\begin{equation}\label{est_D_first_easy}
\left| \langle D_\eps^v,\varphi\chi_\delta\rangle\right|\leq \| D^v_\eps\|_{L^{\sfrac{p}{3}}_{x,t}}\|\varphi \chi_\delta\|_{L^{({\sfrac{p}{3}})'}_{x,t}}\lesssim \eps^{3\theta-1}\delta^{(d-\gamma)\frac{p-3}{p}}\|v\|^3_{L^{p}_t(B^\theta_{p,\infty})}.
\end{equation}
Consider now $D^v *\rho_\eps$ to be the spatial mollification of the Duchon-Robert distribution. We have 
$$
\spt D^v *\rho_\eps\subset \left\{(x+B_\eps(0),t)\,:\, (x,t)\in S \right\}\subset (S)_{\eps,\eps},
$$
from which we deduce 
\begin{equation}\label{support_condit}
\langle D^v *\rho_\eps,\varphi(1-\chi_\delta)\rangle =0\quad \text{if } \eps\leq \frac{\delta}{2}.
\end{equation}
Thus we can rewrite the second term in the right hand side of \eqref{split_D} as 
\begin{align*}
\langle D^v_\eps,\varphi(1-\chi_\delta)\rangle &=\langle D^v_\eps- D^v *\rho_\eps,\varphi(1-\chi_\delta)\rangle \\
&=\langle D^v_\eps- D^v *\rho_\eps,\varphi\rangle -\langle D^v_\eps- D^v *\rho_\eps,\varphi\chi_\delta\rangle.
\end{align*}
Clearly, $\langle D^v_\eps- D^v *\rho_\eps,\varphi\rangle\rightarrow 0$; thus we are left to estimate the term $\langle D^v_\eps- D^v *\rho_\eps,\varphi\chi_\delta\rangle$. 
Notice that by mollifying \eqref{Local_energy} we get
\begin{align}\label{moll_local_energy}
\partial_t \left(\frac{|v|^2_\eps}{2}\right)+\div \left(\left(\frac{|v|^2}{2}+p \right) v \right)_\eps &=-D^v * \rho_\eps.
\end{align}
Then by \eqref{local_energy_eps} and \eqref{moll_local_energy} we can write
\begin{align}
\langle D^v_\eps- D^v *\rho_\eps, \varphi\chi_\delta\rangle&=\left\langle\partial_t\left( \frac{|v|^2_\eps-|v_\eps|^2}{2}\right), \varphi \chi_\delta \right\rangle + \int \div \left((pv)_\eps-p_\eps v_\eps \right) \varphi\chi_\delta\nonumber\\
&\quad+\int \left( \div(R_\eps v_\eps)+\frac{1}{2} \div \left(\left( |v|^2v\right)_\eps - |v_\eps|^2v_\eps \right) \right) \varphi\chi_\delta.\label{eul_all_terms}
\end{align}
By Lemma \ref{l:higher_average} expand 
\begin{equation}\label{cubic_expansion}
(|v|^2v)_\eps =K^v_\eps -2R_\eps v_\eps +(|v|^2_\eps-|v_\eps|^2)v_\eps+ |v_\eps|^2v_\eps,
\end{equation}
where the vector $K_\eps^v$ is the trilinear commutator
$$
K_\eps^v=\int |v(x-h)-v_\eps(x)|^2 (v(x-h)-v_\eps(x))\rho_\eps(h)\,dh.
$$
By plugging \eqref{cubic_expansion} into \eqref{eul_all_terms} we get, by also denoting $P_\eps=(pv)_\eps-p_\eps v_\eps$ the pressure--velocity commutator and by $D_{t,v_\eps}=\partial_t +v_\eps \cdot \nabla$ the advective derivative with respect to $v_\eps$,
\begin{align*}
\langle D_\eps^v -D^v *\rho_\eps,\varphi\chi_{\delta}\rangle &=\left\langle \partial_t\left( \frac{|v|^2_\eps-|v_\eps|^2}{2}\right), \varphi \chi_{\delta} \right\rangle+\int \varphi\chi_{\delta} \div P_\eps  \\
&\quad+ \int  \varphi\chi_{\delta} \div K^v_\eps + \int \varphi\chi_{\delta} \div \left( \frac{|v|^2_\eps-|v_\eps|^2}{2}v_\eps \right) \\
&=-\left\langle D_{t,v_\eps} \frac{\tr R_\eps}{2}, \varphi\chi_{\delta}\right\rangle +\int \varphi\chi_{\delta} \div P_\eps  + \int  \varphi\chi_{\delta} \div K^v_\eps \\
&=E_{tr}+E_{pr}+E_{Re}.
\end{align*}
Thus we have three types of errors to estimate: The transport error $E_{tr}$, the pressure error $E_{pr}$ and the Reynolds error $E_{Re}$, the most delicate being $E_{tr}$. Indeed, the different assumptions we made on the time (Lagrangian/Eulerian) stability of the dissipation set are specifically designed to handle that term.

By using \eqref{CET_commutator}, together with the double regularity of the pressure \eqref{p_double_reg}, we have 
\begin{equation}\label{est_II}
|E_{pr}|\lesssim \| (pv)_\eps-p_\eps v_\eps  \|_{L^{\sfrac{p}{3}}_t(W^{1,{\sfrac{p}{3}}}_x)}\| \varphi\chi_\delta\|_{L^{({\sfrac{p}{3}})'}_{x,t}}\lesssim \eps^{3\theta-1}\delta^{(d-\gamma)\frac{p-3}{p}}\|v\|_{L^{p}_t(B^\theta_{p,\infty})}^3.
\end{equation}
Moreover, by Proposition \ref{p:cubic_comm} we also have
\begin{equation}
\label{eul_Rey_err}
|E_{Re}|\leq \|\div K_\eps^v\|_{L^{\sfrac{p}{3}}_{x,t}} \|\varphi\chi_{\delta}\|_{L^{(\sfrac{p}{3})'}_{x,t}}\lesssim \eps^{3\theta-1}\delta^{(d-\gamma)\frac{p-3}{p}}\|v\|_{L^{p}_t(B^\theta_{p,\infty})}^3.
\end{equation}

To estimate $E_{tr}$ we will need to integrate by parts. Since $\varphi$ is smooth and independent of $\delta$,  then the only problematic terms are the ones in which the derivatives hit $\chi_\delta$.  More precisely we can write
\begin{align*}
E_{tr}\simeq  \int(\tr R_\eps) \varphi \partial_t\chi_\delta + \int \varphi  (\tr R_\eps) v_\eps \cdot \nabla \chi_\delta=I+II,
\end{align*}
where we used the symbol $\simeq$ to denote the behaviour of the leading order terms as $\delta \rightarrow 0$.  By \eqref{CET_commutator} and \eqref{est_chi_der}
\begin{equation*}
|I|\lesssim \| \tr R_\eps\|_{L^{\sfrac{p}{3}}_{x,t}}\|\varphi \partial_t\chi_\delta\|_{L^{({\sfrac{p}{3}})'}_{x,t}}\lesssim \eps^{2\theta} \|v\|^2_{L^{\sfrac{2p}{3}}_t(B^\theta_{{\sfrac{2p}{3}},\infty})} \delta^{(d-\gamma)\frac{p-3}{p}-1}
\end{equation*}
and 
$$
|II|\lesssim \| v_\eps \tr R_\eps\|_{L^{\sfrac{p}{3}}_{x,t}}\|\varphi \nabla\chi_\delta\|_{L^{({\sfrac{p}{3}})'}_{x,t}}\lesssim \eps^{2\theta} \|v\|^3_{L^p_t(B^\theta_{p,\infty})} \delta^{(d-\gamma)\frac{p-3}{p}-1},
$$
from which we deduce
\begin{align}\label{est_III}
|E_{tr}|\lesssim  \eps^{2\theta}\delta^{(d-\gamma)\frac{p-3}{p}-1} \left( \|v\|^3_{L^{p}_t(B^\theta_{p,\infty})}+ \|v\|^2_{L^{p}_t(B^\theta_{p,\infty})}\right).
\end{align}

By combining now \eqref{est_D_first_easy}, \eqref{est_II}, \eqref{eul_Rey_err} and \eqref{est_III} we conclude 
\begin{equation}
\label{final_bound_D}
\left|\langle D^v_\eps,\varphi\rangle \right|\lesssim  \eps^{3\theta-1}\delta^{(d-\gamma)\frac{p-3}{p}}+ \eps^{2\theta}\delta^{(d-\gamma)\frac{p-3}{p}-1}.
\end{equation}
By choosing $\delta=\eps^\alpha$ and optimizing in $\alpha$ we obtain $\alpha=1-\theta$. Notice that this choice is consistent with the condition $\eps\leq \delta$ in \eqref{support_condit}. Thus inserting $\delta=\eps^{1-\theta}$ in \eqref{final_bound_D} we obtain 
\begin{equation}\label{eulerian_final_bound}
\left|\langle D^v_\eps,\varphi\rangle \right|\lesssim \eps^{\left(\frac{2\theta}{1-\theta}-1+(d-\gamma)\frac{p-3}{p} \right)(1-\theta)},
\end{equation}
which concludes the proof by letting $\epsilon\rightarrow 0$.
\qed

\begin{remark}[Eulerian theorem with $v\in L^q_t(B^\theta_{p,\infty})$]
\label{r:eulerian_L3_in_time}
Following exactly the same lines of the previous proof, it is not difficult to see that the bound \eqref{eulerian_final_bound}, and thus the very same conclusion of Theorem \ref{t:main}, can be achieved by assuming $v\in L^q_t(B^\theta_{p,\infty})$, for any two (possibly different) $p,q\in [3,\infty]$ and slightly strengthening the assumption \eqref{E_Min_dim_gamma} by asking
\begin{equation}
    \label{Eulerian_generalised}
    \mathcal{H}^d\big( M\cap (S)_{\delta,\delta}(t)\big)\lesssim g(t) \delta^{d-\gamma},\quad \text{for some } g\in L^{r}([0,T]), \, \text{with } r=\frac{p-3}{p}\frac{q}{q-3}, 
\end{equation}
 where we denoted by
 $$
 (S)_{\delta,\delta}(t):=\big\{x\in \Omega\,:\, (x,t)\in (S)_{\delta,\delta} \big\}
 $$
 the $t-$time slice of the set $(S)_{\delta,\delta}\subset \Omega\times (0,T)$. Note that since we are on a bounded interval $(0,T)$ we can always suppose, without loosing in generality, that $q\leq p$, so that $r\geq 1$. It readily follows that in general \eqref{Eulerian_generalised} is stronger than \eqref{E_Min_dim_gamma}, while by Fubini's theorem the two coincide if $q=p$, since in this case $g\in L^1([0,T])$, thus recovering the same assumptions used in Theorem \ref{t:main}. For $q=3$ --- that is, for solutions $v\in L^3_t(B^\theta_{p,\infty})$ --- we must have $g\in L^\infty([0,T])$. In this case the requirement \eqref{Eulerian_generalised} reads as a (quite strong) uniform in time spatial upper Minkowski dimension at most $\gamma$: This is equivalent to ask that when covering $S$ with space--time cubes of size $\delta$, then every $t-$time slice of that covered set sees at most $\delta^{-\gamma}$ of such cubes.
\end{remark}

\subsection{Proof of Theorem \ref{t:main_lagr_dim}}  We now prove Theorem~\ref{t:main_lagr_dim}.  The main difference in the proof is how the transport term is handled in order to obtain a better exponent. This requires a more accurate definition of the cut--off function used to localise the Duchon-Robert distribution on the dissipative set $S$.

Let $\delta>0$, and $\tau\leq \delta^{1-\theta+\sfrac{(d-\gamma)}{p}}$. Pick a smooth cut--off function $ \tilde \chi_{\delta,\tau}\in C^\infty_c\left( \mathcal{L}^{V^\delta}(S)_{4\delta,4\tau} \right)$ such that 
\begin{equation}
\label{indicator}
 \tilde \chi_{\delta,\tau}\big|_{\mathcal{L}^{V^\delta}(S)_{2\delta,2\tau}}\equiv 1,
\end{equation}

where $\mathcal{L}^{V^\delta}(S)_{2\delta,2\tau}$ is the Lagrangian tubular neighbourhood defined in \eqref{Lagran_neigh}.
Define $\chi_{\delta,\tau}$ by mollifying $\tilde \chi_{\delta,\tau}$ along the flow map $\Phi^{V^\delta}$ (defined in \eqref{flow_def}), i.e.
$$
\chi_{\delta,\tau}(x,t)=\int \tilde \chi_{\delta,\tau}\left( \Phi^{V^\delta}_s(x,t)\right) \eta_\tau(s)\,ds,
$$
where the kernel $\eta_\tau$ is a standard $1-$dimensional Friedrichs mollifier with support in $|s| < \tau$.  We remark that this mollification technique also plays a role in convex integration \cite{Is17holder} and second derivative estimates for suitable weak solutions to Navier-Stokes \cites{VY21, Y20}, but the present use appears to be the first application in the context of Onsager singularity theorems.

Let $D_\eps^v=R_\eps:\nabla v_\eps$ be the approximation of the Duchon-Robert distribution from \eqref{DR_measure}. Let $\varphi\in C^\infty_c(\T^d\times (0,T))$ be a test function.  We want to show that $\langle D_\eps^v,\varphi\rangle\rightarrow 0$ as $\eps \rightarrow 0$. We split it as 
\begin{equation}\label{L_split_D_lagrang}
\langle D_\eps^v,\varphi\rangle=\langle D_\eps^v,\varphi\chi_{\delta,\tau}\rangle+\langle D_\eps^v,\varphi(1-\chi_{\delta,\tau})\rangle.
\end{equation}
  We remark that, although $\varphi \chi_{\delta, \tau}$ is not (at least in the time variable) a smooth test function, we need only linearity and for each term to be well--defined to justify the decomposition ~\eqref{L_split_D_lagrang}.  The subsequent computations show indeed that both of the terms make sense, and thus they a posteriori provide a proof to the validity of such a splitting.

The first term we bound using \eqref{estimate_DRmeasure_approx} and \eqref{L_dim_gamma} together with the fact that $\Phi^{V^\delta}$ is volume preserving
\begin{equation*}
\left| \langle D_\eps^v,\varphi\chi_{\delta,\tau}\rangle\right|\leq \| D^v_\eps\|_{L^{\sfrac{p}{3}}_{x,t}}\|\varphi \chi_{\delta,\tau}\|_{L^{({\sfrac{p}{3}})'}_{x,t}}\lesssim \eps^{3\theta-1}\delta^{(d-\gamma)\frac{p-3}{p}}\|v\|^3_{L^{p}_t(B^\theta_{p,\infty})}.
\end{equation*}
Now consider $D^v *\rho_\eps$, the space mollification of the Duchon-Robert distribution. It is supported in a spatial neighborhood of the dissipation support:
$$
\spt D^v *\rho_\eps\subset \left\{(x+B_\eps(0),t)\,:\, (x,t)\in S \right\}=: (S)_{\eps}.
$$
By the definition of the cut--off $\chi_{\delta,\tau}$ and \eqref{indicator} we have
$$
  \chi_{\delta,\tau} \equiv 1 \quad \text{in a neighborhood of } (S)_{\eps} \text{ if } \, \eps\leq \delta\,,
$$
which implies that 
\begin{equation}
\label{L_support_cond_lagr}
\langle D^v *\rho_\eps,\varphi(1-\chi_{\delta,\tau})\rangle =0\quad \text{if } \eps\leq \delta.
\end{equation}
Thus we can rewrite the second term in the right hand side of \eqref{L_split_D_lagrang} as 
\begin{align*}
\langle D^v_\eps,\varphi(1-\chi_{\delta,\tau})\rangle &=\langle D^v_\eps- D^v *\rho_\eps,\varphi(1-\chi_{\delta,\tau})\rangle \\
&=\langle D^v_\eps- D^v *\rho_\eps,\varphi\rangle -\langle D^v_\eps- D^v *\rho_\eps,\varphi\chi_{\delta,\tau}\rangle.
\end{align*}
Clearly $\langle D^v_\eps- D^v *\rho_\eps,\varphi\rangle\rightarrow 0$, thus we are left to estimate the second term.   By using \eqref{local_energy_eps}, \eqref{moll_local_energy} and Lemma \ref{l:higher_average} as in the proof of Theorem \ref{t:main}  given above, we can write
\begin{align*}
\langle D_\eps^v -D^v *\rho_\eps,\varphi\chi_{\delta,\tau}\rangle &=-\frac{1}{2}\left \langle D_{t,v_\eps} \tr R_\eps, \varphi\chi_{\delta,\tau}\right\rangle+\int \varphi\chi_{\delta,\tau} \div P_\eps  + \int  \varphi\chi_{\delta,\tau} \div K_\eps \\
&=E_{tr}+E_{pr}+E_{Re}.
\end{align*}
where $P_\eps=(pv)_\eps-p_\eps v_\eps$ is the pressure--velocity commutator, $K^v_\eps$ is the cubic commutator from Lemma \ref{l:higher_average} and  $D_{t,v_\eps}=\partial_t +v_\eps \cdot \nabla$  denotes the advective derivative with respect to $v_\eps$. 

By \eqref{e:cubic_comm} and since $\Phi^{V^\delta}$ is volume preserving, we have 
\begin{equation}
\label{L_Rey_err_lagr}
|E_{Re}|\leq \|\div K_\eps^v\|_{L^{\sfrac{p}{3}}_{x,t}} \|\varphi\chi_{\delta,\tau}\|_{L^{\frac{p}{p-3}}_{x,t}}\lesssim \eps^{3\theta-1}\delta^{(d-\gamma)\frac{p-3}{p}},
\end{equation}
and similarly, by using the double regularity of the pressure \eqref{p_double_reg} together with \eqref{CET_commutator}, we also get
\begin{equation}
\label{L_press_err_lagr}
|E_{pr}|\leq \|\div P_\eps\|_{L^{\sfrac{p}{3}}_{x,t}} \|\varphi\chi_{\delta,\tau}\|_{L^{\frac{p}{p-3}}_{x,t}}\lesssim \eps^{3\theta-1}\delta^{(d-\gamma)\frac{p-3}{p}}.
\end{equation}
We are now only left to estimate $E_{tr}$.  We split it as
\begin{equation}\label{E_transp_split_lagr}
-E_{tr}=\frac{1}{2}\left \langle D_{t,V^\delta} \tr R_\eps, \varphi\chi_{\delta,\tau}\right\rangle+\frac{1}{2}\int \varphi\chi_{\delta,\tau}\left(v_\eps-V^\delta \right)\cdot \nabla\left( \tr R_\eps\right) =I+II.
\end{equation}
By using \eqref{moll1}, \eqref{CET_commutator} and \eqref{L_stab_beta1}, we estimate
\begin{align}\label{e:transp_2_lagr}
|II|&\leq \|v_\eps-V^\delta\|_{L^p_{x,t}} \|\nabla(\tr R_\eps)\|_{L^{\sfrac{p}{2}}_{x,t}} \|\varphi\chi_{\delta,\tau}\|_{L^{\frac{p}{p-3}}_{x,t}}\nonumber \\
&\lesssim \left( \|v_\eps-v\|_{L^p_{x,t}}+\|v-V^\delta\|_{L^p_{x,t}}\right) \eps^{2\theta-1}\delta^{(d-\gamma)\frac{p-3}{p}}\|v\|^2_{L^p_t(B^\theta_{p,\infty})}\nonumber \\
&\lesssim \left( \eps^\theta+\delta^\theta\right) \eps^{2\theta-1}\delta^{(d-\gamma)\frac{p-3}{p}}.
\end{align}
To estimate $I$ in \eqref{E_transp_split_lagr} we will need to integrate it by parts. Since $\varphi$ is smooth independently on $\delta$, the only terms that need to be estimated are the ones in which the derivatives hit $\chi_{\delta,\tau}$. More precisely we can write 
$$
I\simeq \int \varphi \tr R_\eps  D_{t,V^\delta} \chi_{\delta,\tau},
$$
where we used the symbol $\simeq$ to denote the leading order term as $\delta \rightarrow 0$.  

We now compute the advective derivative of the cutoff $\chi_{\delta, \tau}$ using the basic property of mollification along the flow.  Because $D_{t,V^\delta}$ commutes with its own flow map we have
\begin{align}\label{advect_cutoff}
D_{t,V^\delta} \chi_{\delta,\tau}&=\frac{d}{d\sigma} \chi_{\delta,\tau}\left(\phi^{V^\delta}_\sigma (x,t),t+\sigma \right)\bigg|_{\sigma=0}\nonumber\\
&= \frac{d}{d\sigma}\int \tilde \chi_{\delta,\tau}\left(\phi^{V^\delta}_s\left(\phi^{V^\delta}_\sigma (x,t),t+\sigma \right),t+\sigma+s \right) \eta_\tau(s)\,ds \bigg|_{\sigma=0}\nonumber\\
&= \frac{d}{d\sigma} \int \tilde \chi_{\delta,\tau}\left(\phi^{V^\delta}_{s+\sigma} (x,t), t+\sigma+s \right)\eta_\tau(s)\,ds \bigg|_{\sigma=0}\nonumber\\
&= \frac{d}{d\sigma} \int \tilde \chi_{\delta,\tau}\left(\phi^{V^\delta}_{s} (x,t), t+s \right) \eta_\tau(s - \sigma)\,ds \bigg|_{\sigma=0}\nonumber\\
&=-\int \tilde \chi_{\delta,\tau}\left( \Phi^{V^\delta}_s (x,t) \right) \eta'_\tau(s)\,ds,
\end{align}

Using the fact that $\Phi^{V^\delta}$ is volume preserving we now obtain
$$
\left\| D_{t,V^\delta} \chi_{\delta,\tau}\right\|_{L^q_{x,t}}\lesssim \tau^{-1} \left\|\tilde \chi_{\delta,\tau} \right\|_{L^q_{x,t}}\lesssim \tau^{-1}\delta^{(d-\gamma)\frac{1}{q}}.
$$
Thus, by the usual quadratic commutator estimate \eqref{CET_commutator},  we get
\begin{equation}
\label{L_transp_err_lagr}
|I|\lesssim \| \tr R_\eps\|_{L^{\sfrac{p}{2}}_{x,t}} \left\| D_{t,V^\delta} \chi_{\delta,\tau}\right\|_{L^{\frac{p}{p-2}}_{x,t}}\lesssim  \tau^{-1}\eps^{2\theta}\delta^{(d-\gamma)\frac{p-2}{p}}.
\end{equation}

Thus by putting \eqref{L_Rey_err_lagr}, \eqref{L_press_err_lagr}, \eqref{e:transp_2_lagr} and \eqref{L_transp_err_lagr} all together we achieve 
$$
|\langle D_\eps^v -D^v *\rho_\eps,\varphi\chi_{\delta,\tau}\rangle|\lesssim \left(\eps^{3\theta-1}+\eps^{2\theta-1}\delta^\theta+\tau^{-1}\eps^{2\theta}\delta^{\sfrac{(d-\gamma)}{p}} \right)\delta^{(d-\gamma)\frac{p-3}{p}},
$$
from which by first choosing $ \delta=\eps$ and then $\tau=\delta^{1-\theta+\sfrac{(d-\gamma)}{p}}$, we deduce
$$
|\langle D_\eps^v -D^v *\rho_\eps,\varphi\chi_{\delta,\tau}\rangle|\lesssim \eps^{3\theta-1+(d-\gamma)\frac{p-3}{p}}.
$$
Note that the choice of $\delta$ is consistent with the support condition \eqref{L_support_cond_lagr} above, which played a crucial role for our computations.
Finally, what we have proved is that 
\begin{equation}
    \label{lagrangian_final_bound}
|\langle D_\eps^v ,\varphi\rangle|\lesssim \eps^{3\theta-1+(d-\gamma)\frac{p-3}{p}},
\end{equation}
which concludes the proof by letting $\eps\rightarrow 0$. 
\qed

\begin{remark}[Lagrangian theorem with $v\in L^q_t(B^\theta_{p,\infty})$]
\label{r:lagrangian_L3_in_time}
Similarly to what we have discussed in Remark \ref{r:eulerian_L3_in_time} about the Eulerian theorem, also in the latter Lagrangian proof the bound \eqref{lagrangian_final_bound}, and thus the very same conclusion of Theorem \ref{t:main_lagr_dim}, can be equivalently achieved by assuming $v\in L^q_t(B^\theta_{p,\infty})$, for any two (possibly different) $p,q\in [3,\infty]$, and by slightly strengthening the assumption \eqref{L_dim_gamma} with
\begin{equation}
    \label{Lagrangian_generalised}
    \mathcal{H}^d\left(M\cap \mathcal{L}^{V^\delta}(S)_{\delta,\tau}(t)\right)\lesssim g(t) \delta^{d-\gamma},\quad \text{for some } g\in L^{r}([0,T]), \, \text{being } r=\frac{p-3}{p}\frac{q}{q-3}, 
\end{equation}
 where 
 $$
 \mathcal{L}^{V^\delta}(S)_{\delta,\tau}(t):=\left\{x\in \Omega\,:\, (x,t)\in \mathcal{L}^{V^\delta}(S)_{\delta,\tau}\right\}
 $$
is the $t-$time slice of the set $\mathcal{L}^{V^\delta}(S)_{\delta,\tau}\subset \Omega\times (0,T)$. Without loosing in generality assume $q\leq p$, so that $r\geq 1$. In general \eqref{Lagrangian_generalised} is stronger than \eqref{L_dim_gamma}, while by Fubini's theorem the two coincide if $q=p$, since in this case $g\in L^1([0,T])$, thus providing back the same assumptions used in Theorem \ref{t:main_lagr_dim}. For $q=3$ --- that is, for solutions $v\in L^3_t(B^\theta_{p,\infty})$ --- we must have $g\in L^\infty([0,T])$. In this case the requirement \eqref{Lagrangian_generalised} reads as a (quite strong) uniform in time spatial Lagrangian upper Minkowski dimension at most $\gamma$: This is equivalent to ask that when covering $S$ with Lagrangian space--time cylinders of spatial radius $\delta$, and the corresponding time length $\tau=\delta^{1-\theta+\sfrac{(d-\gamma)}{p}}$, then every $t-$time slice of that covered set sees at most $\delta^{-\gamma}$ of such cylinders.

\end{remark}
\begin{remark}[External force]\label{r:ext_force}
From the two previous proofs it is easy to see that the presence of an external force of the kind $f\in L^{p'}_{x,t}$, being $p'=\sfrac{p}{(p-1)}$ the H\"older conjugate of $p$, does not affect the analysis, and thus one can reach the very same conclusions of Theorem  \ref{t:main} and Theorem \ref{t:main_lagr_dim}. Indeed the forced version of the local energy balance \eqref{Local_energy} reads as 
$$
\partial_t\left( \frac{|v|^2}{2}\right)+\div \left(\left(\frac{|v|^2}{2}+p \right) v \right)=f\cdot v-D^v.
$$
Since $v\in L^p_{x,t}$ and $f\in L^{p'}_{x,t}$,  then $f\cdot v\in L^1_{x,t}$ and thus automatically defines a distribution. Moreover,  in the two previous proofs this extra term will generate  errors of the kind 
$$
E_{f}:=\int \left(f_\eps\cdot v_\eps-(f\cdot v)_\eps\right)\varphi \chi_{\delta},
$$
where $\chi_{\delta}$ is the space--time cutoff localised around the dissipation sets. More precisely, in the proof of Theorem \ref{t:main} is the cutoff given by Lemma \ref{l:eulerian_cutoff}, while in the proof of Theorem \ref{t:main_lagr_dim} is the cutoff obtained by mollifying along the flow of $V^\delta$ for a time interval of length $\tau=\delta^{1-\theta+\sfrac{(d-\gamma)}{p}}$. In both cases, since $|\chi_\delta|\leq 1$, we can estimate
$$
|E_f|\lesssim \| f_\eps\cdot v_\eps-(f\cdot v)_\eps \|_{L^1_{x,t}}\rightarrow 0,
$$
as $\eps\rightarrow 0$, uniformly in $\delta>0$.
\end{remark}


\section{Vanishing viscosity}\label{s:vanish_visc_proof}
The two proofs of Theorems \ref{t:eul_vanish_visc} and \ref{t:lagr_vanish_visc} are a (very similar) standard contradiction argument but we prefer to give the full details of both of them for the reader's convenience. 
\subsection{Proof of Theorem \ref{t:eul_vanish_visc}} By contradiction assume that \eqref{inter_norm_blowup_euler} does not hold. Thus we can find a (non--relabelled) subsequence $\{v^\nu\}_{\nu>0}$ such that 
$$
\left\| v^\nu\right\|_{L^p_t(B^\theta_{p,\infty})}\leq C.
$$
In particular, since $\theta>0$, by the Aubin-Lions-Simon Lemma we can extract a further subsequence  such that $v^\nu\rightarrow v$ in $L^p_{x,t}$ and moreover $v\in L^p_t(B^\theta_{p,\infty})$. Thus $v$ is a weak solution of incompressible Euler which, since $p\geq 3$,  by \eqref{DR_equals_diss} has dissipation measure $D^v=\mu$, where $\mu$ is the limit in measure of the dissipation $\nu|\nabla v^\nu|^2$ as in the assumptions. By \eqref{0th_law_new} we get that $D^v$ is nontrivial, which then contradicts the fact that $D^v\equiv 0$ by Theorem \ref{t:main}, since by assumption 
$$
\frac{2\theta}{1-\theta}>1-\frac{p-3}{p}(d-\gamma).
$$
\qed
\subsection{Proof of Theorem \ref{t:lagr_vanish_visc} } By contradiction assume that \eqref{inter_norm_blowup_lagr} does not hold. Thus we can find a (non--relabelled) subsequence $\{v^\nu\}_{\nu>0}$ such that 
\begin{equation}\label{uniform_bound_interm_norm}
\left\| v^\nu\right\|_{L^p_t(B^\theta_{p,\infty})}\leq C.
\end{equation}
 In particular the $L^2_{x,t}$ limit $v$ of the sequence $\{v^\nu\}_{\nu>0}$ belongs to $L^p_t(B^\theta_{p,\infty})$. Moreover, since $\theta>0$, the Aubin-Lions-Simon Lemma improves the $L^2_{x,t}$ convergence to the $L^p_{x,t}$ one of the whole subsequence which satisfies \eqref{uniform_bound_interm_norm}. Clearly $v$ solves Euler and by \eqref{DR_equals_diss}, since $p\geq 3$, we also have $D^v=\mu$. Thus by \eqref{0th_law_new} we deduce that $D^v$ has to be nontrivial, which then contradicts Theorem \ref{t:main_lagr_dim}, since we assumed
$$
\theta>\frac13-\frac{p-3}{3p}(d-\gamma).
$$
\qed

\end{document}